\documentclass[a4paper,11pt]{article}
\usepackage[T1]{fontenc}
\usepackage[utf8x]{inputenc} 
\usepackage{lmodern}
\usepackage{amsmath}
\usepackage{amsthm}
\usepackage{amssymb}
\usepackage{authblk}
\usepackage{graphicx}
\usepackage{enumitem}
\usepackage{xcolor}
\usepackage{dsfont}
\usepackage{cleveref}

\usepackage{amsthm}
\usepackage{amssymb}
\usepackage{xparse}
\usepackage{thmtools}
\usepackage{stackrel}
\usepackage{bbold}
\usepackage{relsize}
\usepackage{tikz}

\newcommand{\R}{\mathbb{R}}
\newcommand{\C}{\mathbb{C}}
\newcommand{\Z}{\mathbb{Z}}
\newcommand{\N}{\mathbb{N}}
\newcommand{\Q}{\mathbb{Q}}

\newcommand{\dd}{\mathrm{d}}

\renewcommand{\S}{\mathbb{S}}

\DeclareMathOperator{\Hess}{Hess}

\DeclareMathOperator{\diag}{diag}
\DeclareMathOperator{\Sp}{Sp}

\DeclareMathOperator{\dist}{dist}

\makeatletter
\newtheoremstyle{indented}
{7pt} 
{7pt} 
{} 
{1.5em} 
{\bfseries} 
{.} 
{.5em} 
{} 

\theoremstyle{definition}

\newtheorem{defn}{Definition}[section]

\theoremstyle{plain}
\newtheorem*{theorem*}{Theorem}

\newtheorem{theorem}{Theorem}

\newenvironment{preuve}{
	\noindent \textbf{Proof. }}{\hfill $\square$\medskip\par}

\newtheorem{prop}[defn]{Proposition}

\newtheorem{lem}[defn]{Lemma}

\theoremstyle{definition}
\newtheorem{rem}[defn]{Remark} 

\makeatletter
\renewcommand*\env@matrix[1][*\c@MaxMatrixCols c]{%
  \hskip -\arraycolsep
  \let\@ifnextchar\new@ifnextchar
  \array{#1}}
\makeatother

\usepackage[text={18cm,23cm},centering]{geometry}

\title{WKB eigenmode construction for analytic Toeplitz operators}
\author{Alix Deleporte\thanks{alix.deleporte@universite-paris-saclay.fr}}
\affil{Universit\'e de Strasbourg, CNRS, IRMA UMR 7501, F-67000
  Strasbourg, France\\ \vspace{1em}
  Institut f\"ur Mathematik, Universit\"at Z\"urich,
  Winterthurerstrasse 190. CH-8057 Z\"urich\\  \vspace{1em}
Universit\'e Paris-Saclay, CNRS, Laboratoire de math\'ematiques d'Orsay, F-91405, Orsay, France.}

\usepackage{bbm}

\newcommand\blfootnote[1]{%
  \begingroup
  \renewcommand\thefootnote{}\footnote{#1}%
  \addtocounter{footnote}{-1}%
  \endgroup
}

\begin{document}

\maketitle

\blfootnote{This work was supported by grant
  ANR-13-BS01-0007-01}
\begin{abstract}
  We provide almost eigenfunctions for Toeplitz operators with
  real-analytic symbols, at the bottom of non-degenerate wells. These
  quasimodes follow the WKB ansatz; the error is
  $O(e^{-cN})$, where $c>0$ and $N\to +\infty$ is the inverse
  semiclassical parameter.
\end{abstract}

\section{Introduction}
\label{sec:introduction}

This article is concerned with Berezin-Toeplitz quantization. We
associate, to a real-valued function $f$ on a compact K\"ahler
manifold $M$, a sequence
of self-adjoint operators $(T_N(f))_{N\geq 1}$ acting on spaces
of sections over $M$. These operators are called \emph{Toeplitz operators}. Examples of Toeplitz operators are spin
systems (where $M$ is a product of two-spheres), which
are indexed by the total spin $S=\frac N2$. Motivated
by questions arising in the physics literature about the behaviour of
spin systems at low temperature, we wish to study the lowest-lying
eigenvalues and associated eigenvectors of
Toeplitz operators in the limit $N\to +\infty$. In this article we
specifically study exponential estimates, that is, approximate
expressions with $O(e^{-cN})$ remainder for some $c>0$.

Given $f:M\to \R$, we say that $P_0\in M$ is an elliptic point when
$\nabla f(P_0)=0$ and all eigenvalues of the Hessian of $f$ at $P_0$ are
nonzero and have the same sign. Elliptic points are always local
extrema, while local extrema generically are elliptic points.

We provide, in the special case
where $f$ is real-analytic and has an elliptic point at
$P_0\in M$, a construction of quasimodes for $T_N(f)$: we build (Theorem \ref{thr:WKB-Bottom}) a sequence of normalised sections $(v(N))_{N\geq
  1}$ and a real sequence $(\lambda(N))_{N\geq 1}$, with asymptotic expansions in decreasing powers of $N$, such
that
\[
  T_N(f)v(N)=\lambda(N)v(N)+O(e^{-cN}).
\]
The sequence $v(N)$ takes the form of a Wentzel-Kramers-Brillouin
(WKB) ansatz: it is written as
\begin{equation}\label{eq:WKB_ansatz}
  v(N):x\mapsto C N^{\frac{\dim(M)}{2}}\psi^{\otimes N}(x)(v_0(x)+N^{-1}v_1(x)+\ldots),
\end{equation}
where the \emph{symbol} $(v_k)_{k\geq 0}$ is a sequence of functions
on $M$ that are holomorphic in a neighbourhood $V$ of $P_0$,
and the \emph{phase} $\psi$ is a section over $M$, holomorphic on $V$, and decaying away from $P_0$:
\[
  \exists \epsilon>0, \forall x\in M,|\psi(x)|\leq e^{-\epsilon\dist(x,P_0)^2}.
\]
The multiplicative factor $CN^{\frac{\dim(M)}{2}}$ then ensures that $v(N)$ is
normalised.

Since $T_N(f)$ is self-adjoint, the existence of a quasimode implies that
$\lambda(N)$ is exponentially close to the spectrum of $T_N(f)$, but
not necessarily that $v(N)$ is exponentially close to an
eigenfunction. In Theorem \ref{thr:WKB-Bottom}, we also prove that, if
$f$ is Morse (all critical points have non-degenerate Hessian), the
eigenvectors associated with the lowest eigenvalue of $T_N(f)$ are
exponentially close to a finite sum of quasimodes
of the form \eqref{eq:WKB_ansatz}, attached to the elliptic points
corresponding to global minima of $f$.

\subsection{Bergman kernels and Toeplitz operators}
\label{sec:toeplitz-operators}
Let us rapidly present the basic definitions associated with semiclassical
Berezin-Toeplitz quantization as introduced in full generality in \cite{bordemann_toeplitz_1994}; see the in-depth introductions
\cite{charles_aspects_2000,le_floch_brief_2018}.

Let $(M,\omega)$ be a compact boundaryless symplectic manifold. Berezin-Toeplitz
quantization associates, to a function $f:M\to \R$, a
sequence of Toeplitz operators $(T_N(f))_{N\geq 1}$. To perform this quantization, we have to provide a supplementary geometrical information: a
complex structure $J$, which encodes a notion of holomorphic
objects on $M$, and which is compatible with $\omega$ : $(M,\omega,J)$ is a
Kähler manifold.

\begin{defn}\label{def:toep}
  Let $(M,\omega,J)$ be a compact K\"ahler manifold. Let $L$ be a complex line
  bundle over $M$, and let $h$ be a Hermitian metric on $L$  such that
  $\mathrm{curv}\,h=2i\pi \omega$. The couple $(L,h)$ exists if and only if the
  integral of $\omega$ over each closed surface in $M$ is an integer
  multiple of $2\pi$. We then say that $M$ is \emph{quantizable}. 

 Let $N\in \N$. The \emph{Bergman projector} $S_N$ is the orthogonal projector, from
  the space of square-integrable sections $L^2(M,L^{\otimes N})$ to
  the finite-dimensional subspace of holomorphic sections $H_0(M,L^{\otimes N})$.

  Let also $f:M\to \R$. The Toeplitz operator $T_N(f)$
  associated with $f$ is the following operator:
  \[
    \begin{array}{rcl}T_N(f):H_0(M,L^{\otimes N})&\to &
                                                        H_0(M,L^{\otimes
                                                        N})\\
      u&\mapsto & S_N(fu).
    \end{array}
  \]
\end{defn}
It is convenient to extend $T_N(f)$ into an operator on
$L^2(M,L^{\otimes N})$ by the formula
\[
  T_N(f)=S_NfS_N;
\]
in this way, $(T_N(f))_{N\in \N}$ is a family of finite rank self-adjoint operators.

Given a
Hilbert basis $(s_1,\ldots,s_{d_N})$ of $H_0(M,L^{\otimes N})$, the
Bergman projector $S_N$ admits the following integral kernel:
\[
  S_N(x,y)=\sum_{i=1}^{d_N}s_i(x)\otimes \overline{s_i(y)}.
\]
The study of the Bergman kernel as $N\to +\infty$ lies at the core of
the semiclassics of Toeplitz quantization. In a previous article
\cite{deleporte_toeplitz_2018}, we developed a semiclassical machinery
in real-analytic regularity, in
order to give asymptotic formulas up to an exponential remainder for $S_N$, and Toeplitz operators,
in the case where the symplectic form $\omega$ and the function $f$ are real-analytic on the
complex manifold $(M,J)$. The analysis of the Bergman kernel in
real-analytic geometry is a recent and active topic \cite{berman_direct_2008,hezari_off-diagonal_2017,rouby_analytic_2018,charles_analytic_2021,hezari_property_2021,deleporte_direct_2022}.

\begin{defn}
 \label{def:coh-state-anal} Let $(M,\omega,J)$ be a compact quantizable K\"ahler manifold and
  $(S_N)_{N\geq 1}$ be the associated sequence of Bergman projectors. Let
  $x\in M$ and $N\in \N$. The \emph{coherent state} $\psi^N_x$ at $x$ is the element of
  $H_0(M,L^{\otimes N})\otimes \overline{L}_x^{\otimes N}$ given by freezing the
  second variable of the Bergman kernel: for every $y\in M$, one has
  \[
    \psi^N_x(y)=S_N(y,x).
    \]
\end{defn}

\begin{theorem}
  \label{thr:WKB-Bottom}
Let $M$ be a quantizable compact real-analytic K\"ahler manifold.
Let $f$ be a real-analytic, real-valued function on $M$.
\begin{enumerate}
  \item Let
  $P_0\in M$ be an elliptic point of
  $f$ which is a local minimum. 
  Then there exist
  \begin{itemize}
  \item positive constants $C,c,c',R,\epsilon$,
  \item a neighbourhood $V$ of $P_0$,
  \item a holomorphic function $\varphi$ on $V$ such that
    \[\exists \epsilon\in (0,1),\forall
        x\in V,\,|\varphi(x)|\leq (1-\epsilon)\frac{d(x,P_0)^2}{2}
      ,\]
      \item a sequence of holomorphic functions $(u_k)_{k\geq 0}$ on $V$, with
        $u_0(P_0)=1$ and $u_k(P_0)=0$ for $k\neq 0$, satisfying
        \begin{equation}\label{eq:growth_symbol}
          \forall k\geq 0,\,\sup_V|u_k|\leq CR^kk!,
        \end{equation}
      \item a real sequence $(\lambda_k)_{k\geq 0}$, where
        $\lambda_0=f(P_0)$ and where $\lambda_1$
        is the ground state energy of the quantization of the Hessian of $f$ at $P_0$ (see
  \cite{deleporte_low-energy_2016}), satisfying
  \begin{equation}\label{eq:growth_eigval}
    \forall k\geq 0,\,|\lambda_k|\leq CR^kk!,
  \end{equation}
\end{itemize}
such that, for every $N\geq 1$, if $\psi^{N}_{P_0}$ denotes the coherent state at $P_0$,
  then with
  \begin{align*}
    u(N)&=\mathds{1}_V\psi^N_{P_0}e^{N\varphi}\left(\sum_{k=0}^{cN}N^{-k}u_k\right)\\
    \lambda(N)&=\sum_{k=0}^{cN}\lambda_kN^{-k},
  \end{align*}
  one has
  \[
    \left\|T_N(f)u(N)-\lambda(N)u(N)\right\|_{L^2(M,L^{\otimes
        N})}\leq Ce^{-c'N}.
  \]
  \item
  If the minimal set of $f$ consists in a finite number of
  non-degenerate minimal points, then any normalised eigenfunction of $T_N(f)$
  with minimal eigenvalue is at distance $Ce^{-c'N}$ from a linear
  combination of the functions constructed in item 1 at each
  minimal point.
\end{enumerate}
\end{theorem}
The coherent state $\psi_{P_0}^N$ has a WKB-type expansion (see
Proposition \ref{prop:Szeg-gen}), and one can recover the section
$\psi$ in \eqref{eq:WKB_ansatz} from there and $\varphi$. The analytic symbol $(v_k)_{k\geq
  0}$ is then obtained by normalising $u(N)$, an operation which preserves
the growth property \eqref{eq:growth_symbol}. Thus the expression of
$u(N)$ above implies \eqref{eq:WKB_ansatz}.

Our method of proof
consists in constructing  $\varphi$, satisfying a
Hamilton-Jacobi type equation, then solve by induction a transport
equation on the coefficients $u_k$, and finally to prove the analytic
growth controls \eqref{eq:growth_symbol} and \eqref{eq:growth_eigval}. 

The pseudodifferential equivalent of Theorem \ref{thr:WKB-Bottom} is claimed in
\cite{martinez_microlocal_1999}
, however
all details are not given: the growth property (\cref{eq:growth_symbol,eq:growth_eigval}), which is crucial
to the ability to sum until $k=cN$ terms in \eqref{eq:WKB_ansatz}, is stated without
proof. The verification of estimates of this nature is often non
trivial, and in this case, it is the subject of Propositions \ref{prop:transport} and
\ref{prop:Wells-WKB}. The purpose of this article is not merely to fix
the gap in the strategy proposed in
\cite{martinez_microlocal_1999}, but to extend it to the more general
setting of Berezin-Toeplitz operators.

Indeed, a pseudodifferential
operator on $\R^d$ with real-analytic symbol can be written
\emph{exactly} as a
Toeplitz operator on $M=\C^d$ if the symbol can be extended to a constant width strip
in $\C^d$, since the exact formula for ``Toeplitz to
pseudodifferential'' which can be found for instance in
\cite{zworski_semiclassical_2012}, formula (13.4.12), is a heat-type
evolution at time $N^{-1}=\hbar$ which can be reversed if the
pseudo-differential symbol is analytic. Hence, up to a careful check
of the behaviour at infinity which we do not carry out here, Theorem
\ref{thr:WKB-Bottom} should be enough to provide a complete proof of the result stated in
\cite{martinez_microlocal_1999}.

The Toeplitz point of view on
pseudodifferential operators is relevant for WKB eigenmode
construction and exponential estimates, both from the perspective of
physics \cite{voros_wentzel-kramers-brillouin_1989} and mathematics (at the core of analytic microlocal
analysis is the Fourier-Bros-Iagolnitzer transformation, which relates
pseudodifferential operators to Toeplitz operators). In addition, the Toeplitz setting contains other semiclassical quantum
operators such as spin systems, on which tunnelling estimates are
widely studied in the physics community
\cite{owerre_macroscopic_2015}, although not always in a rigorous way.

WKB estimates for low-energy eigenfunctions in a Morse energy
landscape are well-known for purely
electric Schrödinger operators, of the form
$-\hbar^2\Delta+V(x)$. Without analyticity assumptions on $V$, one can
construct a \emph{formal} WKB ansatz \cite{helffer_multiple_1984} : a sequence of functions of the
form
\[
  u_{\hbar}(x)\sim e^{-\frac{\varphi(x)}{\hbar}}\sum_{k=0}^{+\infty}
  \hbar^ku_k(x),
\]
(where $\sim$ denotes formal summation of
classical symbols), which is a $O(\hbar^{\infty})$-quasimode for the
first eigenvalue $\lambda_\hbar$ of the Schrödinger operator, in the
weighted norm associated with $\varphi$:
\[
  \left\|\left(-\hbar^2\Delta+V-\lambda_\hbar\right)u_{\hbar}\right\|_{L²\left(e^{-\frac{\varphi}{\hbar}}\right)}=O(\hbar^{\infty}).
\]
WKB expansions of quasimodes have been the subject of recent activity in the
context of purely magnetic Schrödinger operators, of the form
$(i\hbar\partial+A(x))²$, in an increasing order or generality
\cite{bonnaillie-noel_magnetic_2016,guedes_bonthonneau_wkb_2017,guedes_bonthonneau_exponential_2019,guedes_bonthonneau_magnetic_2020}. In this context, the symbol $f$ reaches a non-degenerate
minimum on a symplectic submanifold of $(\R^{2d},\omega_{st})$, but
subprincipal effects force the ground state to be microlocalised at
one ``miniwell'' (as in
\cite{helffer_puits_1986,deleporte_low-energy_2017}). Contrary to the electric case,
real-analyticity of the magnetic potential $A$ is, most of the time,
necessary to obtain exponential decay of the ground state away from
the miniwell : in \cite{bonnaillie-noel_magnetic_2016,guedes_bonthonneau_wkb_2017,guedes_bonthonneau_exponential_2019,guedes_bonthonneau_magnetic_2020}, one assumes real-analyticity of $A$ and
concludes in a formal WKB expansion. We conjecture that, as in Theorem
\ref{thr:WKB-Bottom}, the coefficients of these formal WKB expansions can be in fact
summed into an analytic symbol.

\begin{rem}\label{rem:Agmon}
  If the minimal set of $f$ consists in several non-degenerate wells,
  then applying the Part 1 of Theorem \ref{thr:WKB-Bottom} at every well yields that the
  actual ground state, which is exponentially close to an orthogonal
  linear combination of quasimodes as above, has Agmon-type
  exponential decay in a neighbourhood of the minimal set, as in \cite{helffer_multiple_1984}.

  Even if the function $\varphi$ can be extended to all of $M$ and yields, formally, exponential
  decay everywhere except at the minimal point, this rate of decay is blurred, not
  only by the error terms in the expression of the Bergman kernel
  (Proposition \ref{prop:Szeg-gen}) but also by the
  fact that we can only sum up to
  $cN$ with $c$ small when summing analytic symbols (see Proposition \ref{prop:anal-symb}), which
  yields a fixed error of order $e^{-c'N}$ with $c'>0$ small. This
  yields a lower bound to the decay rate for the actual ground state, as a
  function of the position, which follows the blue, continuous line in the following picture:
  
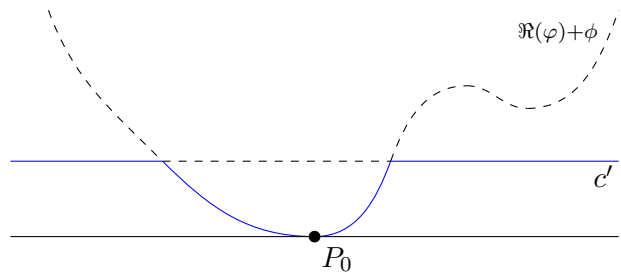
\begin{center}
    \begin{tikzpicture}
       \draw [color=black] (1,1)--(-1,1);
      \draw [color=black] (4,1)--(7,1);
      \draw [style=dashed] (1,1)--(4,1);
      \draw (6.8,-0.2) node {$0$};
      \draw (6.8,0.8) node {$c'$};
      \draw [style=dashed] (-0.5,3) to [out=-70, in=135] (1,1);
      \draw [color=black] (1,1) to [out=-45, in=180] (3,0);
      \draw [color=black] (3,0) to [out=0, in=-110] (4,1);
      \draw [style=dashed] (4,1) to [out=70, in=180] (5,2);
      \draw [style=dashed] (5,2) to [out=0, in=180] (5.8,1.7);
      \draw [style=dashed] (5.8,1.7) to [out=0,in=-110] (7,3);
      \draw (5.7,2.7) node {$\mathsmaller{-\log\left|\psi_{P_0}e^{\varphi}\right|}$};
      \draw [style=dashed] (-1,0)--(7,0);
      \filldraw (3,0) circle (2pt);
      \draw (3.3,-0.3) node {$P_0$};
    \end{tikzpicture}
  \end{center}
  The solid line above is our best estimate for a function $g:M\to
  [0,+\infty)$ such that the ground state $v_N$ of $T_N(f)$ satisfies
  \[
    \forall \epsilon>0,\,\exists C>0, \forall x\in M, \forall N\in \mathbb{N},
    |v_N(x)|\leq Ce^{-(g(x)-\epsilon)N}.
    \]
  
  Near $P_0$, the rate of decay is sharp, but we have
  no explicit control on the constant $c'$.
\end{rem}

Theorem \ref{thr:WKB-Bottom} has applications to tunnelling between
(locally) symmetric wells, in the spirit of 
\cite{helffer_puits_1985}. In Proposition \ref{prop:tunn-lower} we prove that, if $f$ has
two symmetrical wells, and $\lambda_0,\lambda_1$ denote the two first
eigenvalues of $T_N(f)$ (with multiplicity), then
\begin{equation}\label{eq:gap}
  \lambda_1-\lambda_0\leq Ce^{-c'N},
\end{equation}
where $c'$ and $C$ are as in Theorem \ref{thr:WKB-Bottom}. In the physics
community, the tunnelling rate $-N^{-1}\log(\lambda_1-\lambda_0)$ is
often estimated using the degree zero approximation $\varphi$ in the WKB ansatz,
which solves a Hamilton-Jacobi equation (see Proposition
\ref{prop:Hamilton-Jacobi-sol}). However, in Proposition
\ref{prop:tunn-upper}, we provide a series of examples which tend to
illustrate that the tunnelling rate is not given by $\varphi$, and is
not bounded from above by the best possible constant $c'$ in
Theorem \ref{thr:WKB-Bottom}. This contrasts with the case of an
electric Schrödinger operator, where it is well-known that the tunnelling
rate corresponds to the behaviour of the Hamilton-Jacobi equation, as
detailed in
\cite{helffer_puits_1985}. The difference between the two cases is the
ability to extend the problem ``far away'' into the complex, and in
particular, to prove sharp exponential decay, as we explained in Remark
\ref{rem:Agmon}.

Let us now discuss possible alternative strategies for the proof of
Theorem \ref{thr:WKB-Bottom}. The method we follow is the most direct
one, inspired from the $C^{\infty}$ case, and proceeds by a sequence of
perturbations of the Toeplitz operator with a quadratic symbol
corresponding to the Hessian (for which the ground state is
explicit). The necessary verification that the terms of the
perturbation sum into an analytic symbol, i.e. controls
\eqref{eq:growth_symbol} and \eqref{eq:growth_eigval}, occupies
most of the proof.

In some situations, it might be easier to prove that one can conjugate
$T_N(f)$ (microlocally and up to an exponentially small error) into an
operator for which the eigenfunctions are explicit, such as a
quadratic operator. One can hope to do so when $M$ has complex
dimension 1, or more generally for integrable systems near elliptic points. This fact is
used, for instance, in appendix B of
\cite{helffer_semi-classical_1989} concerned with pseudo-differential
operators on $\R$, and leads to a result similar to
Theorem \ref{thr:WKB-Bottom}, with a shorter and simpler
proof. In this integrable case, if one can build a quantum
action-angle theorem near an elliptic point in the analytic category
(which remains to be done), one could describe all eigenfunctions and
eigenvectors modulo an exponentially small error, not just the ground
state.

Apart from the complete integrability assumption, there is hope that KAM-like theorems can be of use,
and more precisely, that under a suitable genericity assumption
on the symplectic diagonalisation of the Hessian, a Birkhoff normal
form near the elliptic point is enough to describe the spectrum, but
it is not clear whether it would provide a simpler proof of Theorem \ref{thr:WKB-Bottom}.

\subsection{Outline}
\label{sec:outline}

In Section \ref{sec:calc-analyt-toepl} we briefly present the tools
which we developed in \cite{deleporte_toeplitz_2018} to tackle problems from semiclassical
analysis in real-analytic regularity in the context of
Berezin-Toeplitz quantization. We then proceed to the proof of
Theorem \ref{thr:WKB-Bottom}.

Section \ref{sec:geometry-wkb-ansatz}
recalls the geometrical ingredients required in order to build a
\emph{formal} WKB ansatz, that is, for every $K\in \N$, a quasimode of the form \eqref{eq:WKB_ansatz}, where there are $K$
terms inside the parenthesis and which satisfies the eigenvalue
equation up to $O(N^{-K-2})$. Each of the coefficients $u_k$ solves a
transport equation, with a source term depending on
$u_0,\ldots,u_{k-1}$. The main novel result of Section
\ref{sec:geometry-wkb-ansatz} is a control the solution of this transport
equation, in an analytic norm, by the source term.

In Section \ref{sec:almost-eigenvectors}, we prove that the
sequences $(u_k)_{k\geq 0}$ and $(\lambda_k)_{k\geq 0}$ 
belong to an analytic class. In particular, they satisfy the growth
condition (\cref{eq:growth_symbol,eq:growth_eigval}).
This allows us to perform an analytic summation and produce a sequence
of sections (indexed by $N$)
which satisfy the eigenvalue equation for $T_N(f)$ up to
$O(e^{-c'N})$, for some $c'>0$. A standard analysis of the
distribution of low-lying eigenvalues of $T_N(f)$ allows us to conclude
the proof in Section \ref{sec:spectr-estim-at}, where we also discuss
the constant $c'$ in the statement of Theorem \ref{thr:WKB-Bottom}. 



\section{Calculus of analytic Toeplitz operators}
\label{sec:calc-analyt-toepl}

To be able to prove the growth condition (\cref{eq:growth_symbol,eq:growth_eigval}), we use the framework developed in a previous article
\cite{deleporte_toeplitz_2018}, which allowed us to study Toeplitz operators with real-analytic
regularity.

Given two real parameters $r>0,m$, we say that a function $f:U\to \C$ on a smooth
open set $U$ of $\R^d$ belongs to the space $H(m,r,U)$ when there
exists $C>0$ such that, for every
$j\geq 0$, one has
\[
  \|f\|_{C^j(U)}:=\sup_{x\in U}\sum_{|\mu|=j}|\partial^{\mu}f(x)|\leq C\cfrac{r^jj!}{(j+1)^m}.
\]

The minimal $C$ such that the control above is true is a Banach norm
for the space $H(m,r,U)$. Such functions are
real-analytic, and can be extended as holomorphic functions in an tube
of radius proportional to $r^{-1}$ around $U$. Reciprocally, by the Cauchy integral
formula, for all $V\subset\subset U$, every
real-analytic function on $U$ belongs to $H(m,r,V)$, for all $m\in \R$
and for some $r>0$ depending on $\dist(V,\R^d\setminus U)$ and the
radius of analyticity of the function near $V$ (see
\cite{deleporte_toeplitz_2018}, Proposition 2.15).

We will often use, in this article, the pointwise version of the $C^j$
seminorm above:
\[
\|\nabla^j f(x)\|_{\ell^1}:=\sum_{|\mu|=j}|\partial^{\mu}f(x)|.
\]
Generalising the definition of $H(m,r,U)$, we obtain \emph{analytic
  (formal) symbols}.

\begin{defn}\label{def:anal-symb}Let $X$ be a compact real-analytic
  manifold, with real-analytic boundary. We fix a finite set
  $(\rho_V)_{V\in \mathcal{V}}$ of local real-analytic charts on open sets $V$ which
  cover $X$.
  \begin{itemize}
    \item 
  Let $j\geq 0$. The $C^j$ seminorm of a function $f:X\to \C$
  which is continuously differentiable $j$ times is defined as
  \[
    \|f\|_{C^j(X)}=\max_{V\in \mathcal{V}}\|f\circ \rho_V\|_{C^j(V)}=\max_{V\in \mathcal{V}}\sup_{x\in
      V}\sum_{|\mu|=j}|\partial^{\mu}(f\circ \rho_V)(x)|.
  \]

  \item Let $r, R,m$ be positive real numbers. The space of analytic symbols
  $S^{r,R}_m(X)$ consists of sequences $(a_k)_{k\geq 0}$ of real-analytic
  functions on $X$, such that there exists $C\geq 0$ such that, for
  every $j\geq 0, k\geq 0$, one has
  \[
    \|a_k\|_{C^j(X)}\leq C\cfrac{r^jR^k(j+k)!}{(j+k+1)^m}.
  \]

  The norm of an element $a\in S^{r,R}_m(X)$ is defined as the
  smallest $C$ as above; then $S^{r,R}_m(X)$ is a Banach space.
\end{itemize}
\end{defn}
The definition of $S^{r,R}_m(X)$ depends on the chosen atlas, but not in
an essential way: elements of $S^{r,R}_m(X)$ for a given atlas belong
to $S^{r',R'}_{m'}(X)$ for another atlas, with $r',R',m'$ suitably
chosen as a function of $r,R,m$ and the two atlases. 

These analytic classes, which we defined and studied in \cite{deleporte_toeplitz_2018}, are
well-behaved with respect to standard manipulations of functions
(multiplication, change of variables, ...) and, most importantly, with
respect to the stationary phase lemma. Another important property is
the summation of such symbols: if $\hbar$ is a semiclassical parameter
(here $\hbar=N^{-1}$), then for $c>0$ small depending on $R$, the sum
\[
  \sum_{k=0}^{c\hbar^{-1}}\hbar^ku_k
\]
is uniformly bounded as $\hbar\to 0$; in this sum, terms of order
$k=\hbar^{-1}$ are exponentially small, so that the choice of
$c$ has an exponentially small influence on the sum.

\begin{prop}\label{prop:anal-symb}[See \cite{deleporte_toeplitz_2018},
  Propositions 3.6 and 3.8] Let $X$ be a compact real-analytic
  manifold with boundary and fix a real-analytic atlas on $X$.
  \begin{description}
  \item[\hspace{1em}Summation]Let
    $f\in S^{r,R}_m(X)$. Let $c_R=\frac{e}{3R}$. Then
    \begin{enumerate}
    \item The function
      \[
        f(N):x\mapsto \sum_{k=0}^{c_RN}N^{-k}f_k(x)
      \]
      is bounded on $X$ uniformly for $N\in \N$.
    \item For every $0<c_1<c_R$, there exists $c_2>0$ such that
      \[
        \sup_{x\in
          X}\left|\sum_{k=c_1N}^{c_RN}N^{-k}f_k(x)\right|=O(e^{-c_2N}).
        \]
    \end{enumerate}
  \item[\hspace{1em}Cauchy product] There exists $C_0\in \R$ such that
  the following is true.

  Let $r,R\geq
  0$ and $m\geq 4$. For $a,b\in S^{r,R}_m(X)$, let us define
  the Cauchy product of $a$ and $b$ as
  \[
    (a*b)_k=\sum_{i=0}^ka_ib_{k-i}.
    \]
    \begin{enumerate}
      \item
    The space $S^{r,R}_m(X)$ is an algebra for this Cauchy product, that
    is,
    \[
      \|a* b\|_{S^{r,R}_m}\leq
      C_0\|a\|_{S^{r,R}_m}\|b\|_{S^{r,R}_m},
    \]
      Moreover, there exists $c>0$ depending only on $R$ such that
    as $N\to +\infty$, one has
    \[
      (a*b)(N)=a(N)b(N)+O(e^{-cN}).
    \]

    \item Let $r_0,R_0,m_0$ positive and $a\in S^{r_0,R_0}_{m_0}(X)$ with $a_0$
    nonvanishing. Then, for every $m$ large enough depending on $a$,
    for every $r\geq r_02^{m-m_0},R\geq R_02^{m-m_0}$, $a$ is invertible (for the Cauchy product) in
    $S^{r,R}_m(X)$, and its inverse $a^{* -1}$ satisfies:
    \[
      \|a^{* -1}\|_{S^{r,R}_m(X)}\leq 2\min(|a_0|)^{-4}\|a\|^3_{S^{r_0,R_0}_{m_0}(X)}.
    \]
  \end{enumerate}
  
  \end{description}
\end{prop}

\begin{rem}\label{rem:alt_symb_class}
A variant of Definition \ref{def:anal-symb} reads
\[
  \|a_k\|_{C^j}\leq C\frac{r^jR^kj!k!}{(j+k+1)^m};
\]
ultimately, the controls on the symbol $(u_k)_{k\geq 0}$ in Theorem
\ref{thr:WKB-Bottom} will take a mixed form between this and $S^{r,R}_m$, see \eqref{eq:ctrl_u_4.2}
and \eqref{eq:ctrl_u_sharp_4.2}. Other definitions can be found in the
literature, as in
\cite{boutet_de_monvel_pseudo-differential_1967}, equation (1.2), or
\cite{sjostrand_singularites_1982}, chapter 1.
These alternative definitions of analytic symbol spaces are all morally
equivalent (they can be embedded into each other by changing the
values of $r$,$R$,$m$). In practice, one has to choose the convention
which suits the particular combinatorial arguments.
\end{rem}

The summation property in Proposition \ref{prop:anal-symb}, together with the stationary phase lemma,
allows us to study Toeplitz operators up to an exponentially small
error. One of the main results of \cite{deleporte_toeplitz_2018},
proved independently \cite{rouby_analytic_2018}, then simplified in \cite{charles_analytic_2021,deleporte_direct_2022,hezari_property_2021}, is an
expansion of the Bergman kernel on a real-analytic K\"ahler manifold,
with error $O(e^{-c'N})$, in terms of an analytic symbol.
\begin{prop}\label{prop:Szeg-gen}(See \cite{deleporte_toeplitz_2018},
  Theorem A, and \cite{rouby_analytic_2018}, Theorem 3.1)
   Let $M$ be a quantizable compact real-analytic K\"ahler manifold of
  complex dimension $d$. There
  exists positive constants $r,R,m,c,c',C$, a neighbourhood $U$ of the
  diagonal in $M\times M$, a section $\Psi$ of $L\boxtimes \overline{L}$
  over $U$, and an analytic symbol
  $a\in S^{r,R}_m(U)$, holomorphic in the first variable,
  anti-holomorphic in the second variable, such that the Bergman
  kernel $S_N$ on $M$ satisfies, for each $x,y\in M\times M$ and
  $N\geq 1$:
  \[
    \left|S_N(x,y)-\mathds{1}_{(x,y)\in U}\Psi^{\otimes N}(x,y)\sum_{k=0}^{cN}N^{d-k}a_k(x,\overline{y})\right|\leq
    C e^{-c'N}.
    \]
  \end{prop}
Note that the constants $c,c',C$ here are different from that of Theorem \ref{thr:WKB-Bottom}.
  
Similar ideas appear in the literature, and have been successfully applied to the
theory of pseudodifferential operators with real-analytic
symbols. Early results \cite{boutet_de_monvel_pseudo-differential_1967} use a special case of our analytic
classes, when $m=0$; from there, a more geometrical theory of analytic
Fourier Integral operators was developed \cite{sjostrand_singularites_1982}, allowing one to
gradually forget about the parameters $r$ and $R$ when applying the
analytic stationary phase lemma. It is surprising
that the introduction of the parameter $m$, which mimics the
definition of the Hardy spaces on the unit ball, was never considered,
although it simplifies the manipulation of analytic functions (for
instance, the
space $H(m,r,V)$ is stable by product if and only if $m\geq 3$). In
\cite{deleporte_toeplitz_2018} and in the present article, it is
crucial that we are able to choose $m$ arbitrarily large.

Along with the definition of symbol classes in Definition
\ref{def:anal-symb}, we will use another analytic symbol class, which
is a mixture of Definition \ref{def:anal-symb} and Remark
\ref{rem:alt_symb_class}. The basic remark is that, by the Stirling formula,
\[
  \frac{(2k)!}{4^kk!k!}\sqrt{2k+1}\in \left[\tfrac 12,1\right],
\]
and in particular, the following classes of symbols are well-behaved:
\[
  \|a_k\|_{C^j}\leq C_a
  \begin{cases}
    \cfrac{r^jR^kj!k!}{(j+k+1)^m}&\text{ if }j<k\\
    \vphantom{a}&\\
    \cfrac{(r/4)^jR^k(j+k)!}{(j+k+1)^{m-\frac 12}}&\text{ if }j\geq k.
  \end{cases}
\]

We end this section with a technical lemma, which is a refinement of
Lemma 4.6 appearing in \cite{deleporte_toeplitz_2018}, adapted to the
symbol class above.

\begin{lem}\label{lem:new_lemma_4.6}
  Let $U,V,\Lambda$ be domains in $\C^d$ containing $0$. Let
  $\kappa_{\lambda}$ be a biholomorphism from $V$ with image contained
  in $U$, with
  real-analytic dependence on $\lambda\in \Lambda$ and suppose that
  $\kappa_0(0)=0$.

  Let $\kappa:(\lambda,v)\mapsto \kappa_{\lambda}(v)$ and suppose that
  there exists constants $C_{\kappa}, r_0,m_0$ such that, for all
  $j\in \N_0$, one has
  \[
    \|\kappa\|_{C^j(V\times \Lambda)}\leq
    C_{\kappa}\frac{r_0^jj!}{(j+1)^{m_0}}.
  \]
  Then the following is true for all $m\geq m_0$ and all $r\geq r_02^{m-m_0+5}$.
  Let $f$ be a real-analytic function on $U\times \Lambda$ and suppose
  that there exists $C_f$ and $k\geq 0$ such that, for all $j\in \N_0:=\{0,1,2,\ldots\}$,
  \[
    \|\nabla^jf(0,0)\|_{\ell^1}\leq C_f\frac{r^jj!k!}{(j+k+1)^m}
  \]
  and furthermore, for all $j\geq k$,
  \[
    \|\nabla^jf(0,0)\|_{\ell^1}\leq C_f\frac{(r/4)^j(j+k)!}{(j+k+1)^{m-\frac
        12}}.
  \]
  Let $n\leq k$ and $i\leq 2n$. Let $\nabla^i_v$ denote the $i$-th
  gradient over the first set of variables, acting on $V\times
  \Lambda$; then
  \[
    g\mapsto(\lambda\mapsto
    \nabla_v^ig(\kappa_{\lambda}(v),\lambda)_{v=0})\]
  is a differential operator of degree $i$ acting on functions on
  $U\times \Lambda$. Let $(\nabla^i_{\kappa})^{[\leq n]}$ denote the
  truncation of this differential operator to a differential operator
  of degree less or equal to $n$.
  Then, with
  \[
    \gamma=16C_{\kappa}r
  \]
  one has, for every $j\geq 0$,
  \[
    \|\nabla^j(\nabla^i_{\kappa})^{[\leq n]}f(0)\|_{\ell^1}\leq
    i^{d+1}j^{d+1}\gamma^iC_f\frac{r^{j+i}k!}{(i+j+k+1)^{m-\frac 12}} \begin{cases}(i+j)!&\text{ if }i\leq n\\
      \max((n+j)!(i-n)!,j!i!)&\text{ otherwise,}
    \end{cases}
  \]
  and, for every $j\geq k-\min(i,n)$,
  \begin{multline*}
    \|\nabla^j(\nabla^i_{\kappa})^{[\leq n]}f(0)\|_{\ell^1}\leq
    i^{d+1}j^{d+1}\gamma^iC_f\frac{(r/4)^{j+i}k!}{(i+j+k+1)^{m}}\\
    \times \begin{cases}(i+j+k)!&\text{ if }i\leq n\\
      \max((n+j+k)!(i-n)!,(j+k)!i!)&\text{ otherwise.}
    \end{cases}
  \end{multline*}
\end{lem}
\begin{proof}
  We proceed as in Lemma 4.6 in
  \cite{deleporte_toeplitz_2018}. By the Faà di Bruno formula, one has
  \begin{multline}\label{eq:Faa}
    \|\nabla^j((\nabla^i_{\kappa})^{[\leq n]}f)(0)\|_{\ell^1}\leq
    i^dj^d\\ \times
    \sum_{|P|=1}^{\min(n,i)}\sum_{\substack{e_0+\cdots+e_{|P|}=j\\s_1+\cdots+s_{|P|}=|P|}}\frac{j!}{e_0!e_1!\cdots
      e_{|P|}!}\frac{i!}{(|P|)!s_1!\cdots
      s_{|P|}!}\|\nabla^{|P|+e_0}f(0)\|_{\ell^1}\prod_{i=1}^{|P|}\|\kappa\|_{C^{s_i+e_i}}.
  \end{multline}
  We now inject the controls on $f$ and $\kappa$. First of all, for
  all $j_1\in \N_0$,
  \[
    \|\kappa\|_{C^{j_1}}\leq C\frac{(r/32)^{j_1}j_1!}{(j_1+1)^m},
  \]
  and in particular, if $j_1\geq 1$,
  \[
    \|\kappa\|_{C^{j_1}}\leq C\frac{(r/16)^{j_1}(j_1-1)!}{j_1^m}
  \]
  since $2^{j_1}\geq j_1$.

  Injecting this along with the control on $f$, the general term in the sum
  \eqref{eq:Faa}
  is bounded by
  \begin{multline*}
    \frac{j!i!r^{|P|+e_0}(r/16)^{i+j-e_0}(|P|+e_0)!k!(s_1+e_1-1)!\cdots
      (s_{|P|}+e_{|P|}-1)!}{(|P|)!e_0!\cdots e_{|P|}!s_1!\cdots
      s_{|P|}!}\\ \times \frac{1}{(|P|+e_0+k+1)^m(s_1+e_1)^m\cdots (s_{|P|}+e_{|P|})^m}
  \end{multline*}
  and, if $|P|+e_0\geq k$, there holds the more precise bound
  \begin{multline*}
    \frac{j!i!(r/4)^{|P|+e_0}(r/4)^{i+j-e_0}(|P|+e_0+k)!(s_1+e_1-1)!\cdots
      (s_{|P|}+e_{|P|}-1)!}{(|P|)!e_0!\cdots e_{|P|}!s_1!\cdots
      s_{|P|}!}\\ \times \frac{1}{(|P|+e_0+k+1)^{m-\frac 12}(s_1+e_1)^m\cdots (s_{|P|}+e_{|P|})^m}.
  \end{multline*}
  The constraints on $(s_j)$ and $(e_j)$ are such that one can
  simplify the second factors:
  \begin{align*}
    \frac{1}{(|P|+e_0+k+1)^m(s_1+e_1)^m\cdots (s_{|P|}+e_{|P|})^m}&\leq
    \frac{1}{(i+j+k+1)^m}\\
    \frac{1}{(|P|+e_0+k+1)^{m-\frac 12}(s_1+e_1)^m\cdots (s_{|P|}+e_{|P|})^m}&\leq
    \frac{1}{(i+j+k+1)^{m-\frac 12}}.
  \end{align*}
  Moreover, by Lemma 2.5 in \cite{deleporte_toeplitz_2018},
  \[
    \frac{(s_1+e_1-1)!(s_{|P|}+e_{|P|}-1)!}{s_1!\cdots
      s_{|P|}!e_1!\cdots e_{|P|}!}\leq
    \frac{(i-|P|+j-e_0)!}{(i-|P|+1)!(j-e_0)!}.
  \]
  Thus, one has the following general bound on the general term of the
  sum in \eqref{eq:Faa}:
  \[
    C_f(C_{\kappa})^{|P|}\frac{j!i!r^{|P|+e_0}(r/16)^{i+j-e_0}(|P|+e_0)!k!(i-|P|+j-e_0)!}{(|P|)!e_0!(i-|P|+1)!(j-e_0)!}\frac{1}{(i+j+k+1)^{m}},
  \]
  and, provided $|P|+e_0\geq k$, the more precise bound
  \[
    C_f(C_{\kappa})^{|P|}\frac{j!i!(r/4)^{|P|+e_0}(r/16)^{i+j-e_0}(|P|+e_0+k)!(i-|P|+j-e_0)!}{(|P|)!e_0!(i-|P|+1)!(j-e_0)!}\frac{1}{(i+j+k+1)^{m-\frac
        12}}.
  \]
  In both cases, one can isolate
  \[
    \frac{i!}{(|P|)!(i-|P|+1)!}\leq 2^i
  \]
  and
  \[
    \frac{(i-|P|+j-e_0)!}{(j-e_0)!}\leq 2^{i+j-e_0}(i-|P|)!;
  \]
  thus, the general bound simplifies into
  \[
    C_f\frac{(C_{\kappa})^{|P|}r^{|P|}}{8^{j-e_0}4^i}\frac{r^{i+j}(|P|+e_0)!k!(i-|P|)!j!}{e_0!}\frac{1}{(i+j+k+1)^m},
  \]
  and the specific bound into
  \[
    C_f\frac{(C_{\kappa})^{|P|}r^{|P|}}{2^{j-e_0}}\frac{(r/4)^{i+j}(|P|+e_0+k)!(i-|P|)!j!}{e_0!}\frac{1}{(i+j+k+1)^{m-\frac
        12}}.
    \]
  
  Let us now count the number of terms. For fixed
  $|P|-e_0$ and $j$, there are $\binom{i}{|P|}\leq 2^i$ choices for
  $s_1,\ldots,s_{|P|}$ (since each of them must be positive) and
  $\binom{j-e_0+|P|}{|P|}\leq 2^{j-e_0+|P|}$ choices for
  $e_1,\ldots,e_{|P|}$, which are non-negative. Thus, fixing $|P|$ and
  $e_0$ and summing over $s_1,\ldots,s_{|P|},e_1,\ldots,e_{|P|}$, the
  resulting sum is bounded by
  \[
    C_f(C_{\kappa})^{|P|}r^{|P|}\frac{r^{i+j}(|P|+e_0)!k!(i-|P|)!j!}{4^{j-e_0}e_0!}\frac{1}{(i+j+k+1)^m}
  \]
  and, provided $|P|+e_0\geq k$,
  \[
     C_f(C_{\kappa})^{|P|}r^{|P|}2^{i+|P|}\frac{(r/4)^{i+j}(|P|+e_0+k)!(i-|P|)!j!}{e_0!}\frac{1}{(i+j+k+1)^{m-\frac
       12}}.
 \]
 Both formulas above are increasing with respect to $e_0$. If
 $e_0=k-|P|$, moreover, the second formula is larger than the first
 one up to losing a power of $|P|$: indeed, the ratio between the two is
 \[
   \frac{2^{i+|P|}}{4^{i+k-|P|}}\frac{(2k)!}{k!k!}\sqrt{i+j+k+1}\geq \frac{8^{|P|}}{2^{i-1}}.
 \]
 To conclude, if $j+|P|\leq k$, then the sum over $e_0$ is bounded by
 \[
   jC_f(16C_{\kappa})^{|P|}r^{|P|}r^{i+j}(|P|+j)!k!(i-|P|)!\frac{1}{(i+j+k+1)^m}
 \]
 and if $j+|P|\geq k$, then this sum is bounded by
 \[
   jC_f(16C_{\kappa})^{|P|}r^{|P|}2^i(r/4)^{i+j}(|P|+j+k)!(i-|P|)!\frac{1}{(i+j+k+1)^{m-\frac
       12}}.
   \]
   We artificially added the factor $16^{|P|}$ in the first bound so
   that, if $j+|P|\geq k$, then the second bound implies the first
   one.

   We can now conclude: if $j+\min(i,n)\leq k$ (that is to say, if
   $j+|P|$ is always less than $k$), we sum the first bound over
   $|P|$, remarking that it is log-convex with respect to $|P|$. We
   obtain that the sum appearing in \eqref{eq:Faa} is bounded by
   \[
     ijC_f(16C_{\kappa}r)^{i}\frac{r^{i+j}k!}{(i+j+k+1)^m}\max_{|P|\in
       \{0,\min(i,n)\}}(|P|+j)!(i-|P|)!,
   \]
   If $j+\min(i,n)\geq k$, then we
   can apply the second bound for all $|P|$, so that we similarly
   obtain
   \[
     ijC_f(16C_{\kappa}r)^i\frac{(r/4)^{i+j}}{(i+j+k+1)^m}\max_{|P|\in
       \{0,\min(i,n)\}}(|P|+j+k)!(i-|P|)!.
   \]
   This concludes the proof.
 \end{proof}

\section{Geometry of the WKB Ansatz}
\label{sec:geometry-wkb-ansatz}

In this section we provide the geometric ingredients for the proof of Theorem \ref{thr:WKB-Bottom}. We formally
proceed as in the case of a Schr\"odinger operator
\cite{helffer_semi-classical_1988}. If a real-analytic, real-valued
function $f$ has a non-degenerate local
minimum at $P_0\in M$, we seek a sequence of eigenfunctions of
$T_N(f)$ of the
form
\[
  \psi_{P_0}^Ne^{N\varphi}(u_0+N^{-1}u_1+\ldots),
\]
where $\psi_{P_0}^N$ denotes the coherent state at $P_0$. If
$f(P_0)=0$, then the associated sequence
of eigenvalues should be of order $O(N^{-1})$, that is to say, follow
the asymptotic expansion:
\[
  N^{-1}\lambda_0+N^{-2}\lambda_1+\ldots.
\]
When solving the
eigenvalue problem, the terms of order $0$ in
\[e^{-N\varphi}T_N(f)\psi_{P_0}^Ne^{N\varphi}(u_0+N^{-1}u_1+\ldots)\] yield an
equation on $\varphi$. In the case of a Schr\"odinger operator this
is the eikonal equation $|\nabla \varphi|^2=V$, which is solved using
the Agmon metric. In our more general case, we are in presence of a
form of the Hamilton-Jacobi equation (see \eqref{eq:Hamilton-Jacobi} below), which we
solve in Proposition \ref{prop:Hamilton-Jacobi-sol} using a geometric
argument based on the existence of a stable manifold, in the spirit of \cite{sjostrand_analytic_1983}. Associated with $f$ and $\varphi$ are transport equations
which we must solve in order to recover the sequence of functions
$(u_k)_{k\geq 0}$. In Proposition \ref{prop:transport} we study this
transport equation under the point of view of symbol spaces of
Definition \ref{def:anal-symb}. This will allow us, in Proposition
\ref{prop:Wells-WKB}, to perform an analytic summation of the $u_k$'s
in order to find an exponentially accurate eigenfunction for $T_N(f)$,
with exponential decay away from $P_0$.

The plan of this section is as follows: we begin in
Subsection \ref{sec:study-phase-function} with
the study of an analytic phase which will be a deformation of the
phase $\Phi_1$ considered above. We then define and study the Hamilton-Jacobi equation
associated with a real-analytic function near a non-degenerate minimal
point, and the associated transport equations, in Subsections
\ref{sec:Hamilton-Jacobi-equation} and \ref{sec:transport-equations}
respectively. This geometric insight on the construction of a
quasimode attached to an elliptic point is not new, but the purpose of
this section is to fix notations, to present these ideas in a
self-sustained way and in the geometric context of Berezin-Toeplitz
situation, and to prove an analytic estimate for the solution of the
transport equation (Proposition \ref{prop:transport}).

In the rest of this article,
\begin{itemize}
  \item $(M,\omega,J)$ is a quantizable real-analytic compact
    K\"ahler manifold (which means that $\omega$ is real-analytic on
    the complex manifold $(M,J)$); $L,(S_N)_{N\in \N}$ and $(T_N)_{N\in \N}$ are
    the prequantum line bundle, the Bergman projectors and the
    Toeplitz quantizations of Definition \ref{def:toep};
    \item $f$ is a real-valued function on $M$ with
real-analytic regularity.
\item $U_0$ is a small neighbourhood of an elliptic point $P_0$ of $f$ which
  is a local minimum (such
  that the objects below exist on $U_0$); without loss of generality $f(P_0)=0$;
\item $\phi$ is a K\"ahler potential near $U_0$ such that, in a chart
  where $P_0$
  is mapped to $0$,
\[
  \phi(y)=\frac{|y|^2}{2}+O(|y|^3):
\]
that is, $\phi:U_0\to \R$ satisfies
\[
  \partial\overline{\partial}\phi=i\omega;
\]
\item $\widetilde{\phi}$ is the holomorphic function on $U_0\times
  \overline{U_0}$ such
  that $\widetilde{\phi}(x,\overline{x})=\phi(x)$ (holomorphic extension or
  polarisation of $\phi$);
\item More generally, $\widetilde{\phantom{f}}$ represents holomorphic
  extension of real-analytic functions: for instance, $\widetilde{f}$ is the
  extension of $f$ and is defined on $U_0\times \overline{U_0}$;
\item $\Phi_1:U_0^2\times \overline{U_0}^2$ is defined by 
\[
    \Phi_1:(x,y,\overline{w},\overline{z})\mapsto
  2\widetilde{\phi}(x,\overline{w})-2\widetilde{\phi}(y,\overline{w})+2\widetilde{\phi}(y,\overline{z})-2\widetilde{\phi}(x,\overline{z}).
\]
\end{itemize}
The function $\Phi_1$ is associated with the Bergman kernel $S_N$ in
the following way: the section $\Psi$ of Proposition \ref{prop:Szeg-gen} satisfies, for all $(x,y,z)\in U^3$:
  \[
    \langle \Psi^{\otimes N}(x,y),\Psi^{\otimes
      N}(y,z)\rangle_{L^{\otimes N}_y}=\Psi^{\otimes
      N}(x,z)\exp(N\Phi_1(x,y,\overline{y},\overline{z})).
  \]

\setcounter{subsection}{-1}

\subsection{Formal identification of the WKB ansatz}
\label{sec:form-ident-wkb}
We search for an eigenfunction of $T_N(f)$ of the form
\[
  x\mapsto \psi^N_0(x)e^{N\varphi(x)}(u_0(x)+N^{-1}u_1(x)+\ldots),
\]
where $\psi^N_0$ is the coherent state at $0$ (see Definition \ref{def:coh-state-anal}), and
$\varphi,u_0,u_1,\ldots$ are holomorphic functions on a fixed
neighbourhood of $0$.

This construction is local. Indeed, the only situation where the holomorphic functions
$\varphi,u_0,u_1,\ldots$ can be extended to the whole of $M$ is when
they are constant. However, if $\varphi$ does not grow too fast (see
Definition \ref{def:adm-phase}), then the trial function above is exponentially
small outside any fixed neighbourhood of zero.
In particular, applying $T_N(f)$ yields, by Proposition \ref{prop:Szeg-gen},
\begin{multline*}
  T_N(f)(e^{N\varphi}(u_0+N^{-1}u_1+\ldots)\psi^N_0):\\x\mapsto
  \psi^N_0(x)e^{N\varphi(x)}\int_Ue^{N\Phi_1(x,y,\overline{y},0)+N\varphi(y)-N\varphi(x)}f(y)\left(\sum_{k=0}^{cN}
    N^{d-k}a_k(x,y)\right)\left(u_0(y)+N^{-1}u_1(y)+\ldots\right)\dd
  y\\+O(e^{-cN}).
\end{multline*}
If the function appearing in the exponential
\[
  \Phi_2:(x,y)\mapsto \Phi_1(x,y,\overline{y},0)+\varphi(x)-\varphi(y)
\]
is a positive phase
function in the sense of \cite{melin_fourier_1975} (which is
guaranteed if $\varphi$ does not grow too fast, see Proposition
\ref{prop:phase:well}), one can apply the stationary phase lemma
(\cite{sjostrand_singularites_1982}, Theorem 2.8). If $y_*(x)$ is the
critical point of this phase (which belongs to the complexification $\widetilde{U_0}=U_0\times \overline{U_0}$), at dominant order, one has
\[
  T_N(f)(e^{N\varphi}u_0\psi^N_0)(x)=\psi^N_0(x)e^{N\varphi(x)}\widetilde{f}(y_*(x))\widetilde{a_0}(x,y_*(x))\widetilde{u_0}(y_*(x))J(x)+O(N^{-1}).
\]
where $J$ is a non-vanishing Jacobian.

Since we search for an eigenfunction with eigenvalue close to zero, we
want this principal term to vanish. As $J$ and $a_0$ do not vanish,
this yields
\[
 \widetilde{f}(y_*(x))=0,
\]
which boils down to a particular PDE on $\varphi$, the
\emph{Hamilton-Jacobi equation}. We provide a geometric solution to
this equation in Proposition \ref{prop:Hamilton-Jacobi-sol}.

To study the higher orders of the stationary phase lemma we introduce,
as in \cite{sjostrand_singularites_1982}, Lemma 2.7, 
a $x$-dependent, holomorphic change of variables $\kappa_x$, from a
neighbourhood of $y_*(x)$ in $\widetilde{U}_0$ to a neighbourhood of
$0$ in $\C^{2d}$, such that
\begin{equation}\label{eq:defkappa}
  \widetilde{\Phi_2}\circ\kappa_x^{-1}(v_1,\overline{v_2})=v_1\cdot \overline{v_2},
\end{equation}
as well as the associated gradient and Laplacian, acting as follows on
holomorphic functions on $U_0\times \widetilde{U_0}$:
\begin{align}
  \label{eq:defgradkappa}
  (\nabla_{\kappa_x}b):(x,y,\overline{w})&\mapsto
    \left(\frac{\partial b(x,\kappa^{-1}_x(v_1,\overline{v_2}))}{\partial
        v_{1,j}}(x,\kappa_x(y,\overline{w})),\frac{\partial b(x,\kappa^{-1}_x(v_1,\overline{v_2}))}{\partial
        \overline{v_{2,j}}}(x,\kappa_x(y,\overline{w}))\right)_{1\leq j\leq d}\\
  \label{eq:deflaplkappa}
  (\Delta_{\kappa_x}b):(x,y,\overline{w})&\mapsto
  \sum_{j=1}^d\frac{\partial^2b(x,\kappa_x^{-1}(v_1,\overline{v_2})}{\partial
  v_{1,j}\partial \overline{v_{2,j}}}.
\end{align}

At next order, the eigenvalue equation reads, for all $x\in U_0$,
\begin{align*}
  N^{-1}\lambda_0u_0(x)&=T_N(f)(e^{N\varphi}(u_0+N^{-1}u_1)\psi_0^N)(x)\text{
  mod }N^{-2}\\
  &=N^{-1}\psi_0^N(x)e^{N\varphi(x)}\left(\widetilde{f}J(\widetilde{a_0}\widetilde{u_1}+\widetilde{a_1}\widetilde{u_0})(x,y_*(x))
  +\Delta_{\kappa_x}(\widetilde{f}\widetilde{a_0}\widetilde{u_0}J)(x,y_*(x))\right).
\end{align*}
Since $\widetilde{f}(y_*(x))=0$, there is no contribution from $u_1$
at this order. Moreover, one can distribute

\[\Delta_{\kappa_x}(\widetilde{f}\widetilde{a_0}\widetilde{u_0}J)=\widetilde{f}\widetilde{a_0}J\Delta_{\kappa_x}\widetilde{u_0}+\widetilde{u_0}\Delta_{\kappa_x}(\widetilde{f}\widetilde{a_0}J)+\nabla_{\kappa_x}(\widetilde{f}\widetilde{a_0}J)\cdot
  \nabla_{\kappa_x}(\widetilde{u_0}).
\]
The first term of the right-hand side is zero when evaluated at
$y_*(x)$ since $\widetilde{f}(y_*(x))=0$. We obtain
\[
  \left(\nabla_{\kappa_x}(\widetilde{f}\widetilde{a_0}J)\right)(y_*(x))\cdot (\nabla_{\kappa_x}\widetilde{u_0})(y_*(x))=u_0(x)\left(\lambda_0-\Delta_{\kappa_x}(\widetilde{f}\widetilde{a_0}J)(y_*(x))\right).
\]
Observe that $\widetilde{f}$, as the complex extension of $f$, has a
critical point at $x=0$, so that, as long as $y_*(0)=0$ (which is
proved in Proposition \ref{prop:phase:well}), there holds
$\nabla_{\kappa_0}(\widetilde{f}\widetilde{a_0}J)(y_*(0))=0$. Hence, the
equation above implies
\[
  \lambda_0=\Delta_{\kappa_0}(\widetilde{f}\widetilde{a_0}J)(0).
\]
We will see in Proposition \ref{prop:WKB-formal} that this $\lambda_0$
indeed corresponds to the ground state energy of the Hessian of $f$ at
zero.
It remains to solve an equation of the form
\begin{equation}\label{eq:form_transport}
  \left(\nabla_{\kappa_x}(\widetilde{f}\widetilde{a_0}J)\right)(y_*(x))\cdot
  (\nabla_{\kappa_x}\widetilde{u_0})(y_*(x))=u_0(x)h(x),
\end{equation}
where $h(x)=\lambda_0-\Delta_{\kappa_x}(\widetilde{f}\widetilde{a_0}J)(y_*(x))$ vanishes at $x=0$. 
Similar equations are satisfied by the successive terms $u_k$. This
family of equations is solved (with a convenient control on the
size of the solution) in Proposition \ref{prop:transport}. Then, in
Section \ref{sec:almost-eigenvectors} we will prove by induction that the sequence
$(u_k)_{k\geq 0}$ indeed forms an analytic symbol and that the eigenvalue
equation admits a solution up to an $O(e^{-c'N})$ error.

\subsection{A family of phase functions}
\label{sec:study-phase-function}

In this subsection we study a family of analytic phases (in the sense
of Definition 3.11 in \cite{deleporte_toeplitz_2018}) given by a WKB ansatz at the bottom
of a well. To begin with, we describe the conditions on a holomorphic
function $\varphi$ at a neighbourhood of zero, such that $
  e^{N\varphi}\psi_{0}^N$
is a convenient first-order candidate for the ground state of $T_N(f)$.

\begin{defn}
\label{def:adm-phase}A holomorphic function $\varphi$ on $U_0$ is said to be \emph{admissible}
under the following conditions:
\begin{align*}
  \varphi(0)&=0\\
  \nabla\varphi(0)&=0\\
  \exists t<1,\,\forall x\in U_0,\, |\varphi(x)|&\leq \frac t2 |x|^2.
\end{align*}
\end{defn}
\begin{prop}\label{prop:phase:well}
  Let $\varphi$ be an admissible function. The function from $U_0\times U_0$
  to $\C$ defined by
  \begin{equation}\label{eq:defphi2}
    \Phi_2:(x,y)\mapsto \Phi_1(x,y,\overline{y},0) +\varphi(y)-\varphi(x)
  \end{equation}
  is, for all $x$ in a small neighbourhood of zero, a positive phase
  function of $y$ in the sense of \cite{melin_fourier_1975}.

The complex critical point of $\Phi_2$ is $y_*(x)=(x,\overline{y}_c(x))$, where
the holomorphic function $x\mapsto \overline{y}_c(x)$ satisfies
\begin{equation*}
  -2\partial_1\widetilde{\phi}(x,\overline{y}_c(x))+2\partial_1\widetilde{\phi}(x,0)=-
  \partial\varphi(x).
\end{equation*}
In particular, $\overline{y}_c(0)=0$.
\end{prop}
\begin{preuve}
  Near $y=\overline{w}=0$, there holds
  \[
    \Phi_1(0,y,\overline{w},0)=-y\cdot \overline{w}+O(|y,\overline{w}|^3).
    \]

    In particular, for $x=0$, the function $(y,\overline{w})\mapsto
    \widetilde{\Phi_2}(0,y,\overline{w})$ has a critical point at
    $(0,0)$ whose Hessian has a non-degenerate, negative real part (because
    $|\varphi(y)|\leq\frac{t|y|^2}{2}$). In particular, for $x$ small
    enough, $\widetilde{\Phi_2}$ has exactly one
    critical point near $0$, with non-degenerate, negative Hessian
    real part. The
  critical point $(y,\overline{w})$ satisfies the two equations
  \begin{align*}
    \overline{\partial}_{\overline{w}}\widetilde{\phi}(x,\overline{w})-\overline{\partial}_{\overline{w}}\widetilde{\phi}(y,\overline{w})&=0\\
    -2\partial_y\widetilde{\phi}(y,\overline{w})+2\partial_y\widetilde{\phi}(y,0)&=-\partial \varphi(y).    
  \end{align*}
  The first equation yields $y=x$, then the second equation has only
  one solution $\overline{w}=\overline{y}_c(x)$, so that the phase at this critical
  point is equal to
  \[
    2\widetilde{\phi}(x,\overline{y}_c(x))-2\widetilde{\phi}(x,\overline{y}_c(x))+2\widetilde{\phi}(x,0)-2\widetilde{\phi}(x,0)+\varphi(x)-\varphi(x)=0.
  \]
  This concludes the proof.
\end{preuve}

\subsection{Hamilton-Jacobi equation}
\label{sec:Hamilton-Jacobi-equation}

Let $\varphi$ be an admissible function. For every $x\in M$ close to
$0$, there exists one
$\overline{y_c}(x)$ in $\overline{U_0}$ such
that $(x,\overline{y_c}(x))$ is a critical point for the phase of
Proposition \ref{prop:phase:well}.

In order to find the phase of the WKB ansatz, we want to solve, in a neighbourhood of $0$, the following system of
 equations on $\varphi$ and $\overline{y_c}$, where $\varphi$ is an
 admissible function:
 \begin{equation}\label{eq:Hamilton-Jacobi}
   \begin{cases}
     \widetilde{f}(x,\overline{y_c}(x))=0.\\
     -2\partial_1\widetilde{\phi}(x,\overline{y}_c(x))+2\partial_1\widetilde{\phi}(x,0)=-\partial\varphi(x).
\end{cases}
\end{equation}
This will be called the \emph{Hamilton-Jacobi equation}. This equation is
non-trivial already at the formal level: for fixed $x$ the equation
$\widetilde{f}(x,\overline{y})=0$ defines (a priori) a manifold of
complex codimension $1$, which has a singularity at $x=0$. On the other hand,
we need to ensure that
$\partial_1\widetilde{\phi}(x,\overline{y}_c(x))$ is a closed
holomorphic $1$-form
in order to solve for $\varphi$. 

\begin{prop}
  \label{prop:Hamilton-Jacobi-sol}
  The Hamilton-Jacobi equation \eqref{eq:Hamilton-Jacobi} admits
  a solution near $0$, such that $\varphi$ is analytic. 
\end{prop}
\begin{proof}
  We follow the usual method (see the appendix of
  \cite{sjostrand_analytic_1983}), which will consist in considering
  the stable manifold of the Hamiltonian flow of $\widetilde{f}$ for a
  certain symplectic form.
  
  Since the Taylor expansion of $\phi$ at zero is
  \[
    \phi(x)=\frac 12 |x|^2+O(|x|^3),
    \]
    the map
  \[
    \overline{w}\mapsto 2\partial_1\widetilde{\phi}(x,\overline{w})=\overline{w} + O(|x,\overline{w}|^2)
  \]
  is a biholomorphism in a neighbourhood of zero, for $x$ small. Let $\gamma_x$ denote its
  inverse, then $\gamma_x$ is tangent to identity at $x=\overline{w}=0$.

  Letting \[\widetilde{f}_1:(x,\overline{z})\mapsto
    \widetilde{f}(x,\gamma_x(\overline{z})),\] the Hamilton-Jacobi equation
  \eqref{eq:Hamilton-Jacobi} is equivalent to the modified system:
  \[
    \begin{cases}
      \widetilde{f}_1(x,\overline{z_c}(x))=0\\
      -\overline{z_c}(x)+2\partial_1\widetilde{\phi}(x,0)=-\partial \varphi(x).
    \end{cases}
  \]

  The first step is to solve this equation at main order, that is,
  when $\widetilde{f}_1,\widetilde{\phi},\varphi$ are
  quadratic. This can be done using a KAK decomposition, and for
  completeness and pedagogical purposes we detail how this is done.
  The construction of this decomposition will also play a role in the proof
  of Proposition \ref{prop:transport}. 
  
  Let $Q$ be the Hessian of $f$ at zero and $\widetilde{Q}$ its
  holomorphic extension (as a quadratic form). Then $\widetilde{f}_1(x,\overline{z})=\widetilde{Q}(x,\overline{z})+O(|x,\overline{z}|^3)$
  since $\gamma_x$ is tangent to identity at $x=\overline{w}=0$.

  In the modified system, there holds $\overline{z_c}(x)=\partial (2\widetilde{\phi}(x,0)+\varphi(x))$, so that
  finding $x\mapsto \overline{z_c}(x)$ amounts to finding a
  holomorphic Lagrange submanifold $L=\{x,\overline{z_c}(x)\}$ of $\C^d\times \overline{\C^d}$ near $0$,
  for the standard symplectic form $\Im(\sum \dd x_j\wedge \dd \overline{z_j})$
  (which extends the symplectic form $\sum \dd \Re(x_j)\wedge \dd\Im(x_j)$), such that $L$ is
  contained in $\{\widetilde{f}_1=0\}$ and is transverse to the
  vertical $\{x=0\}$. Then, near $0$, one
  has $L=\{x,\partial F(x)\}$ for some holomorphic $F$, and it will only remain to
  check that $\varphi=F-2\widetilde{\phi}(\cdot,0)$ is admissible. As
  in \cite{sjostrand_analytic_1983}, from $f$ and the
  standard symplectic form, the Lagrangean $L$
  will be constructed as the stable manifold of the fixed point $0$
  for the symplectic flow of $\widetilde{f}_1$.

  Suppose $\widetilde{f}_1$ is
  quadratic; that is, $\widetilde{f}_1=\widetilde{Q}$.
  The quadratic form $Q$ admits a symplectic diagonalisation
  with respect to the (real) symplectic form $\sum \dd\Re(x_j)\wedge \dd\Im(x_j)$: there
  exists a symplectic matrix $S$, and positive numbers
  $\omega_1,\ldots, \omega_d$, such that
  \[
    Q=S^T\diag(\omega_1,\omega_1,\omega_2,\omega_2,\ldots,\omega_d,\omega_d)S.
  \]
  Let us study how this symplectic change of variables $S$ behaves under
   complexification. From the $KAK$ decompostion of the semisimple
  Lie group $Sp(2d)$ (or, more practically, using a singular value decomposition), the matrix $S$ can be written as $U_1DU_2$, where
  $U_1$ and $U_2$ belong to $Sp(2d)\cap O(2d)\simeq U(d)$, and $D=\diag(\sigma_1,\sigma_1^{-1},\ldots,
  \sigma_d,\sigma_d^{-1})>0.$

  We now complexify $U_1,U_2,D$ as $\R$-linear endomorphisms of $\C^d$ (in
  contrast with $Q$, which we complexified as a quadratic form). The complexified actions of $U_1$ and $U_2$ are straightforward: for
  $j=1,2$ one
  has $\widetilde{U}_j(x,\overline{z})=(U_jx,U_j^{-1}\overline{z})$. The action of
  $D$ is diagonal: $D=\diag(D_1,\ldots, D_d)$, with
  \[
    D_j(\Re(x_j),\Im(x_j))=\left(\sigma_j\Re(x_j),\sigma_j^{-1}\Im(x_j)\right).
  \]
  Hence, the action of $\widetilde{D}$ is block-diagonal, with
  \[
    \widetilde{D}_j(x_j,\overline{z_j})=\left(\cfrac{\sigma_j+\sigma_j^{-1}}{2}x_j+\cfrac{\sigma_j-\sigma_j^{-1}}{2}\overline{z_j},\cfrac{\sigma_j-\sigma_j^{-1}}{2}x_j+\cfrac{\sigma_j+\sigma_j^{-1}}{2}\overline{z_j}\right).
  \]
  After applying successively the changes of variables
  $\widetilde{U}_1, \widetilde{D},\widetilde{U}_2$, in the new
  variables, the quadratic form becomes
  \[
    \widetilde{f}_1\circ\widetilde{S}:(q,p)\mapsto \sum_{j=1}^d\omega_jq_jp_j.
    \]
Among the zero set of this form, a space of particular interest is
$\{p=0\}$. It is a holomorphic Lagrangean subspace, which is preserved by the symplectic gradient flow of
  $\widetilde{f}_1\circ\widetilde{S}$, and such that every solution starting from this subspace
  tends to zero for positive time. This subspace $\{p=0\}$
  is the \emph{stable manifold of zero} for the symplectic gradient of
  $\widetilde{f}_1\circ\widetilde{S}$. Let us show that, in the starting coordinates
  $(x,\overline{z})$, the stable manifold of $\widetilde{f}_1$ leads
  to an admissible 
  solution of the Hamilton-Jacobi equation.
  \begin{itemize}
    \item
  The inverse change of variables $\widetilde{U}_2^{-1}$ leaves
  $\{p=0\}$ invariant.
\item
  The inverse change of variables $\widetilde{D}^{-1}$ sends
  $\{p=0\}$ to $\{\overline{z}=Ax\}$, with $\|Ax\|_{\ell^2}\leq t\|x\|_{\ell^2}$ for some $t<1$. Indeed,
  the matrix $A$ has diagonal entries
  $\frac{\sigma_j-\sigma_j^{-1}}{\sigma_j+\sigma_j^{-1}}\in (-1,1).$
\item
  The inverse change of variables $\widetilde{U}_1^{-1}$ sends
  $\{\overline{z}=Ax\}$ to $\Lambda_0=\{\overline{z}=U_1AU_1^{-1}x\}$, with a similar property:
  for some $t<1$, there holds $\|U_1AU_1^{-1}x\|_{\ell^2}\leq t\|x\|_{\ell^2}$.
\end{itemize}
Then $\Lambda_0$ is  a linear space of the form $\{\overline{z}=\partial F_0(x)\}$, where $F_0$
is the holomorphic function
  \[
    F_0:x\mapsto \frac 12 \langle x,U_1AU_1^{-1}x\rangle.
  \]
  Hence $\varphi:x\mapsto
  F_0(x)-2\widetilde{\phi}(x,0)=F_0(x)+O(|x|^3)$ is an admissible  solution to the
  Hamilton-Jacobi equations.

  If $\widetilde{f}_1$ is quadratic, we just identified a holomorphic
  Lagrange submanifold transverse to $\{x=0\}$ and contained in
  $\{\widetilde{f}_1=0\}$, as the stable manifold of $0$ for the
  Hamiltonian flow of $\widetilde{f}_1$. In the general case, $\widetilde{f}_1$ is a small
  perturbation of its quadratic part in a small neighbourhood of $0$,
  so that, by the stable manifold Theorem (\cite{ruelle_ergodic_1979},
  Theorem 6.1), the stable subspace $\Lambda_0$ is deformed into a stable manifold $L$
  which has the same properties: $L$ is Lagrangean (since it is a
  stable manifold of a symplectic flow, it must be isotropic, and $L$
  has maximal dimension), and it is
  transverse to $x$ a small neighbourhood of zero since $T_0L$ is the linear
  Lagrangean subspace $\Lambda_0$ described above. Moreover, the
  Hamiltonian flow of $\widetilde{f}_1$ preserves $\widetilde{f}_1$ so that $L$ is
  contained in $\{\widetilde{f}_1=0\}$.

  We finally let $F$ be a holomorphic function such
  that $L=\{(x,\partial F(x))\}$.
  With $\varphi:x\mapsto 
  F(x)-2\widetilde{\phi}(x,0)$, and $\overline{z_c}(x)=\partial F(x)$,
  we obtain a solution to the modified Hamilton-Jacobi equation
    \[
    \begin{cases}
      \widetilde{f}_1(x,\overline{z_c})=0\\
      -\overline{z_c}+\partial_1\widetilde{\phi}(x,0)=-\partial \varphi(x).
    \end{cases}
  \]
Since $\widetilde{\phi}(x,0)=O(|x|^3)$, one has
$\varphi(x)=F(x)+O(|x|^3)=F_0(x)+O(|x|^3)$, so that
\[|\varphi(x)|=|F_0(x)|+O(|x|^3)\leq\frac t2 |x|^2\] for some $t<1$ on a neighbourhood of $0$. This
concludes the proof. 
\end{proof}


\subsection{Transport equations}
\label{sec:transport-equations}
In the proof of Theorem \ref{thr:WKB-Bottom}, one must solve
recursively transport equations of the form \eqref{eq:form_transport},
and prove that the solution is well-controlled. Let us prove that one
can control the solution of this equation by the source term.


\begin{prop}
  \label{prop:transport}
  Let $f':U_0\times \widetilde{U_0}\mapsto \C$ be holomorphic and such that
  \[
    f'(x,y,\overline{w})=\widetilde{f}(y,\overline{w})+O(|x,y,\overline{w}|^3),\]
 and let $\varphi$ be an admissible solution of the Hamilton-Jacobi
  equation \eqref{eq:Hamilton-Jacobi}. Let $x\in U_0$ and let
  $\nabla_{\kappa_x}$ as defined in \eqref{eq:defgradkappa}. 
  Let also $\overline{y}_c$ be the holomorphic function of $x$ such
  that $(x,\overline{y}_c(x))$ is the critical point of $\Phi_2$ as
  defined in \eqref{eq:defphi2}.
  Then there exists $U\subset U_0$ containing $0$ such that the
  following is true.

  For every $g:U\to \C$ holomorphic with $g(0)=0$, and every
  $h:U\to \C$ holomorphic with $h(0)= 0$, there exists a unique
  holomorphic function $u:U\to \C$ with $u(0)=0$ which solves the following
  transport equation:
  \[
    (\nabla_{\kappa_x}f')(x,x,\overline{y}_c(x))\cdot
    (\nabla_{\kappa_x}[(x,y,\overline{w})\mapsto u(y)])(x,x,\overline{y}_c(x))=h(x)u(x)+g(x).
  \]
  
  Moreover, there exists a $\C$-linear change of variables
  $A(f',\varphi)$ on $\C^d$, and positive constants $r_0(h,f',\varphi)$,
  $m_0(h,f',\varphi)$, $C(h,f',\varphi)$ such that,
  for every \[k\geq 0,\quad m\geq m_0(h,f',\varphi),\quad r\geq
  r_0(h,f',\varphi)2^{m-m_0(h,f',\varphi)},\quad C_g>0,\] for every $g$ as above which
satisfies, for every $j\geq 0$,
\begin{equation}\label{eq:ctrl_g_3.5}
    \sum_{|\mu|=j}|\partial^{\mu}(g\circ A(f',\varphi))(0)|\leq
C_g\cfrac{r^j(j+1)!k!}{(1+j+k+1)^m},
\end{equation}
one has, for every $j\geq 0$,
  \begin{equation}\label{eq:ctrl_u_3.5}
    \sum_{|\mu|=j}|\partial^{\mu}(u\circ A(f',\varphi))(0)|\leq
  C(h,f',\varphi)C_g\cfrac{r^jj!k!}{(1+j+k)^m}.
\end{equation}
If moreover $g$ satisfies the sharper control
\begin{equation}\label{eq:ctrl_g_sharp_3.5}
  \sum_{|\mu|=j}|\partial^{\mu}(g\circ A(f',\varphi))(0)|\leq
  C_g\cfrac{(r/4)^j(j+1+k)!}{(1+j+k+1)^{m-\frac 12}}\qquad \qquad \forall j\geq k,
\end{equation}
Then $u$ satisfies
 \begin{equation}\label{eq:ctrl_u_sharp_3.5}
  \sum_{|\mu|=j}|\partial^{\mu}(u\circ A(f',\varphi))(0)|\leq
  C(h,f',\varphi)C_g\cfrac{(r/4)^j(j+k)!}{(1+j+k)^{m-\frac 12}}\qquad \qquad \forall j\geq
  k.
\end{equation}
\end{prop}
Note that \eqref{eq:ctrl_g_sharp_3.5} is sharper than
\eqref{eq:ctrl_g_3.5}, and similarly \eqref{eq:ctrl_u_sharp_3.5} is sharper than
\eqref{eq:ctrl_u_3.5}, because for every $j\geq k$ one has
\[
  \frac{(j+k)!\sqrt{j+k+1}}{j!k!4^j}\leq 1;
\]
in fact in the limit case $k=j$ one has
\[
  \frac{(2j)!\sqrt{2j+1}}{j!j!4^j}\in \left[\frac 12,1\right].
  \]

\begin{proof}

  We let $X$ be the vector field on $U$ such that
  \[
    (\nabla_{\kappa_x}f')(x,x,\overline{y}_c(x))\cdot
    (\nabla_{\kappa_x}[(x,y,\widetilde{w})\mapsto
    u(y)])(x,x,\overline{y}_c(x))=X\cdot u(x).
  \]
  The proof consists in four steps. In the first step we prove that
  all trajectories of $X$ converge towards $0$ in negative time, so that there is no
  dynamical obstruction to the existence of $u$ near $0$ (if $X$ had wandering
  or closed trajectories, solving $X\cdot u=fu+g$ would require
  specific conditions on $f$ and $g$). 
  In the second step, we identify the successive terms
  of a formal power expansion of $u$, which allows us to control
  successive derivatives of $u$ at $0$: that is, we prove inequality
  \eqref{eq:ctrl_u_3.5} using \eqref{eq:ctrl_g_3.5}. In the third step, we prove
  that the solution $u$ is well-defined on the whole of $U$ for some
  small neighbourhood $U$ of $0$, using the
  Duhamel formula. In the fourth step, we
  finally prove that \eqref{eq:ctrl_g_sharp_3.5}$\Rightarrow$\eqref{eq:ctrl_u_sharp_3.5}.

  {\bf First step}
  
  We study the dynamics of the vector field $X$ in a
  neighbourhood of zero. To this end, we relate $\kappa_x$ to the linear
  change of variables which appeared in the proof of Proposition
  \ref{prop:Hamilton-Jacobi-sol} in the case where $f$ is quadratic.

  We first note that, as the Taylor expansion of $f'$ is 
  \[f'=\widetilde{f}+O(|x,y,\overline{w}|^3)=O(|x,y,\overline{w}|^2),\] one has
  $X(0)=0$. The Hessian of $\varphi$ at zero is determined by the Hessian of $f$
  at zero; it then determines the linear part of $\kappa_0$ at $0$,
  hence the linear part of $X$ at $0$. Up to a linear unitary change
  of variables, there exists a
  diagonal matrix $A$, a unitary matrix $U_1$, and positive $\omega_1,\ldots,\omega_d$, such that
  \[
    f:x\mapsto \sum_{j=1}^d\omega_j|(U_1Ax)_j|^2+O(|x|^3).
  \]
  Then $\varphi(x)= \frac 12\langle x, U_{1}AU_1^{-1}x\rangle+O(|x|^3)$, so that the phase reads
  \[
    \Phi_1(x,y,\overline{w},0)+\varphi(y)-\varphi(x)=2(x-y)\cdot\left(\overline{w}-\frac 14
    U_1AU_1^{-1}(x+y)\right)+O(|x,y,\overline{w}|^3).
    \]
    In particular, at first order, one can write
    \[
      \kappa_x(y,\overline{w})=\left(y-x,\overline{w}-\frac 14
        U_1^{-1}AU_1(y+x)\right)+O(|x,y,\overline{w}|^2).
    \]
    Hence, the inverse change of variables is of the form
    \[
      \kappa_x^{-1}(v_1,\overline{v_2})=\left(v_1+x,\overline{v_2}+\frac
        14U_1^{-1}AU_1(v_1+2x)\right)+O(|x,v_1,\overline{v_2}|^2),
    \]
    so that the restriction to the diagonal
    \[
      u\circ\kappa_x^{-1}(v,\overline{v})=u(v+x)+O(|x,v,\overline{v}|^2)
    \]
    is holomorphic with respect to $v$, at first order.

    We then wish to compute
    \[
      \nabla_{\kappa_x}f'\cdot
      \nabla_{\kappa_x}u:=[\overline{\partial}_v(f'\circ
      \kappa_x^{-1})\cdot \partial_v(u\circ
      \kappa^{-1}_x)+\partial_v(f'\circ\kappa_x^{-1})\cdot
      \overline{\partial}_v(u\circ \kappa_x^{-1})](v_1=\overline{v_2}=0)
    \]
    which is equal, at first order, to the opposite symplectic flow (for
    the symplectic form $\Im(\dd v\wedge \dd\overline{v})$) of $\widetilde{f}$
    applied to $u$:
    \[
      \nabla_{\kappa_x}f'\cdot
      \nabla_{\kappa_x}u=[\overline{\partial}_v(\widetilde{f}\circ
      \kappa_x^{-1})\cdot \partial_v(u\circ
      \kappa^{-1}_x)-\partial_v(\widetilde{f}\circ\kappa_x^{-1})\cdot
      \overline{\partial}_v(u\circ \kappa_x^{-1})](v_1=\overline{v_2}=0)+O(|x|^2),
    \]
    since $\partial_v(\tilde{f}\circ \kappa_x^{-1})\cdot
    \overline{\partial}_v(u\circ \kappa_x^{-1})=O(|x|^2)$.
    
    As seen in the proof of Proposition
    \ref{prop:Hamilton-Jacobi-sol}, the critical manifold
    $\{v_1(y,\overline{w})=\overline{v_2}(y,\overline{w})=0\}$ is the stable manifold for the Hamiltonian
    flow of $\widetilde{f}$, so that each trajectory of the vector
    field above is repulsed from zero in a non-degenerate way.

    {\bf Second step.}

    Since $X$ has $0$ as non-degenerate repulsive point, it can be
    diagonalised: there exists a linear change of variables
    $A(f',\varphi)$ on $\C^d$
    after which
    \[
      X=\sum_{i=1}^d\omega_ix_i\partial_{x_i}+O(|x|^2),
    \]
    for positive $\omega_i$.
  From now on we apply this linear change of variables and we will
  control $\|\nabla^ju(0)\|_{\ell^1}$ in these coordinates, from
  $\|\nabla^jg(0)\|_{\ell^1}$ in the same coordinates.

  Note that, by
  the Poincaré-Dulac theorem, after a non-linear change of variables,
  the non-linear $O(|x|^2)$ part in $X$ commutes with the linear part;
  this additional simplification is not needed here. Note also that,
  generically, the $\omega_i$'s are independent over $\Q$. In
  this case, in principle, one could completely eliminate the non-linear part in
  $X$, and in particular, build WKB quasimodes corresponding
  to a higher eigenvalue, not only the microlocal ground state.
  
  Let us expand
  \begin{align*}
    X\cdot u(x)&=\sum_{i=1}^d\left(\omega_ix_i+\sum_{|\nu|\geq 2}\frac{a_{i,\nu}}{\nu!}x^{\nu}\right)\cfrac{\partial}{\partial
                        x_i}\,u(x)\\
    h(x)&=\sum_{|\nu|\geq 1}\frac{h_{\nu}}{\nu!}x^{\nu}\\
    g(x)&=\sum_{|\nu|\geq 1}\frac{g_{\nu}}{\nu!}x^{\nu}.
  \end{align*}

  Then, for some $V\subset \subset U_0$ which contains $0$, for some
  positive $r_0,m_0$, one has $a_i\in H(m_0,r_0,V)$ and
  $h\in H(m_0,r_0,V)$, so that
  \begin{align*}
|h_{\nu}|&\leq  C_h\frac{r_0^{|\nu|}\nu!}{(1+|\nu|)^{m_0}}\qquad
           \forall |\nu|\geq 1\\
    |a_{i,\nu}|&\leq
                 C_a\frac{r_0^{|\nu|-1}\nu!}{|\nu|^{m_0}}\qquad
    \forall |\nu|\geq 2.
  \end{align*}
  The index shift on the control of $a_{i}$ will balance the one in \eqref{eq:induc_u_3.5} below.
  
  Let $m\geq m_0$ and $r\geq r_02^{2+m-m_0}$, to be fixed later
  on. Then, one has also
  \begin{align}
|h_{\nu}|&\leq  C_h\frac{(r/4)^{|\nu|}\nu!}{(1+|\nu|)^{m}}\qquad
           \forall |\nu|\geq 1 \label{eq:ctrl_h_gen}\\
    |a_{i,\nu}|&\leq C_a\frac{(r/4)^{|\nu|-1}\nu!}{|\nu|^{m}}\qquad \forall |\nu|\geq 2\label{eq:ctrl_a_gen}.
  \end{align}
  
  Let us now suppose that \eqref{eq:ctrl_g_3.5} holds, that is, for some $k\geq 0$, for every $j\geq 0$, one has
  \begin{equation}\label{eq:ctrl_g_rec_3.5}
    \sum_{|\nu|=j}|g_{\nu}|\leq C_g\frac{r^{j}k!(j+1)!}{(1+k+j+1)^m}.
  \end{equation}
  We will solve the transport equation with
  \[
    u:x\mapsto\sum_{|\nu|\geq 1}\cfrac{u_{\nu}}{\nu!}\,x^{\nu},
  \]
  and prove by induction on $j\geq 0$ that \eqref{eq:ctrl_u_3.5}
  holds, i.e. 
  \begin{equation}\label{eq:ctrl_u_rec_3.5}
    \sum_{|\mu|=j}|u_{\mu}|\leq C(h,f',\varphi)C_g\frac{r^jk!j!}{(1+k+j)^m},
  \end{equation}
  as long as $m$ is large enough with respect to $C_a$ and $C_h$, and
  $r$ is large enough accordingly.
  
  For $j=0$, one has $u(0)=0$ by hypothesis.
  The transport equation is equivalent to the following family of equations indexed by
  $\mu$ with
  $|\mu|\geq 1$:
  \begin{equation}\label{eq:induc_u_3.5}
    u_{\mu}\frac{\sum_{i=1}^d\omega_i\mu_i}{\mu!}=\sum_{|\nu|\geq
      1}\cfrac{h_{\nu}u_{\mu-\nu}}{\nu!(\mu-\nu)!}+\cfrac{g_{\mu}}{\mu!}-\sum_{i=1}^d\sum_{|\nu|\geq
      2}\cfrac{a_{i,\nu}u_{\mu-\nu+\eta_i}}{\nu!(\mu-\nu+\eta_i)!}.
  \end{equation}
  Here, as in the rest of the proof, $\eta_i$ denotes the base
  polyindex with coefficients $(0,0,\ldots, 0,1,0, \ldots, 0)$ where
  the $1$ is at the site $i$.
  
  Observe that $u_{\mu}$ appears only on the left-hand side of the
  equation above, while the
  right-hand side contains coefficients $u_{\rho}$ with $\rho<\mu$.
  As the eigenvalues $\omega_i$ are all positive, one can solve for
  $u_{\mu}$ by induction. Indeed, there exists $C_{\omega}>0$ such
  that, for every $|\mu|\neq 0$ there holds \[\sum_{i=1}^d\omega_i\mu_i\geq
  C_{\omega}^{-1}(|\mu|+1).\] In particular,
  \[
    |u_{\mu}|\leq \frac{C_{\omega}}{|\mu|+1}\left(|g_{\mu}|+\left|\sum_{|\nu|\geq
        1}\cfrac{h_{\nu}u_{\mu-\nu}\mu!}{\nu!(\mu-\nu)!}\right|+\left|\sum_{i=1}^d\sum_{|\nu|\geq
        2}\cfrac{a_{i,\nu}u_{\mu-\nu+\eta_i}\mu!}{\nu!(\mu-
      \nu+\eta_i)!}\right|\right).
\]
One has, directly, from \eqref{eq:ctrl_g_rec_3.5},
\[
  \sum_{|\mu|=j}\frac{C_{\omega}|g_{\mu}|}{|\mu|+1}\leq
  C_gC_{\omega}\frac{r^{j}k!j!}{(1+k+j)^m}.
  \]
From \eqref{eq:ctrl_h_gen}, one has
\begin{align*}
  \sum_{|\mu|=j}\left|\sum_{|\nu|\geq
  1}\cfrac{h_{\nu}u_{\mu-\nu}\mu!}{\nu!(\mu-\nu)!}\right|&\leq\sum_{\ell=1}^{j-1}\sum_{|\rho|=\ell}|u_{\rho}|\sum_{\substack{|\mu|=j\\\mu\geq\rho}}\cfrac{|h_{\mu-\rho}|}{(\mu-\rho)!}\cfrac{\mu!}{\rho!}\\
  &\leq
    C_h\sum_{\ell=1}^{j-1}r^{j-\ell}\sum_{|\rho|=\ell}|u_{\rho}|\sum_{\substack{|\mu|=j\\\mu\geq\rho}}\cfrac{\mu!}{\rho!}\cfrac{1}{(1+j-\ell)^m}.
\end{align*}
Note that, when applying \eqref{eq:ctrl_h_gen}, we have loosened $(r/4)^j$ into $r^j$ ; the
supplementary power $4^j$ will be used only in the fourth step.

For $|\rho|=\ell$ there holds
\[
  \sup_{\substack{|\mu|=j\\\mu\geq\rho}}\cfrac{\mu!}{\rho!}\leq
  \cfrac{j!}{\ell!},
\]
since if $\rho_{M}$ denotes the largest index of $\rho$ the supremum
above is $(\rho_M+1)(\rho_M+2)\ldots(\rho_M+j-\ell)$.
Moreover, there
are less than $(j-\ell+1)^d$ polyindices $\mu$ such that $|\mu|=j$ and
$\mu\geq\rho$ with $|\rho|=\ell$.

Hence, by the induction hypothesis \eqref{eq:ctrl_u_rec_3.5},
\begin{align*}
  \sum_{|\mu|=j}\left|\sum_{|\nu|\geq
  1}\cfrac{h_{\nu}u_{\mu-\nu}\mu!}{\nu!(\mu-\nu)!}\right|&\leq
                                                           C_h\sum_{\ell=1}^{j-1}r^{j-\ell}\cfrac{j!}{\ell!}\cfrac{(1+j-\ell)^d}{(1+j-\ell)^{m}}\sum_{|\rho|=\ell}|u_{\rho}|\\
  &\leq
    C_hC(h,f',\varphi)C_g\cfrac{r^jj!k!}{(1+k+j)^m}\sum_{\ell=1}^{j-1}\cfrac{(1+j-\ell)^d(1+k+j)^m}{(1+j-\ell)^m(1+k+\ell)^m}.
\end{align*}

From Lemma 2.13 in \cite{deleporte_toeplitz_2018}, if $m\geq \max(d+2,2(d+1))$,
there holds
\[
  \sum_{\ell=1}^{j-1}\cfrac{(1+j-\ell)^d(1+k+j)^m}{(1+j-\ell)^m(1+k+\ell)^m}\leq
  C(d)\cfrac{3^m}{4^m}.
\]
In particular,
\[
   \sum_{|\mu|=j}\left|\sum_{|\nu|\geq
  1}\cfrac{h_{\nu}u_{\mu-\nu}\mu!}{\nu!(\mu-\nu)!}\right|\leq
C_hC(d)\cfrac{3^m}{4^m}\,C(h,f',\varphi)C_g\cfrac{r^jj!k!}{(1+k+j)^m}.\]
For $m$ large enough with respect to $C_hC(d)C_{\omega}$, and $r\geq r_02^{2+m-m_0}$, one has\[\sum_{|\mu|=j}\left|\sum_{|\nu|\geq
  1}\cfrac{h_{\nu}u_{\mu-\nu}\mu!}{\nu!(\mu-\nu)!}\right|\leq \frac 1{3C_{\omega}}
C(h,f',\varphi)C_g\cfrac{r^jj!k!}{(1+k+j)^m}.\]

Similarly, from \eqref{eq:ctrl_a_gen}, one can control, for $1\leq i \leq d$, the quantity
\begin{align*}
  \left|\sum_{|\nu|\geq
        2}\cfrac{a_{i,\nu}u_{\mu-\nu+\eta_i}\mu!}{\nu!(\mu-
      \nu+\eta_i)!}\right|&\leq\sum_{\ell=1}^{j-1}\sum_{|\rho|=\ell}|u_{\rho}|\sum_{\substack{|\mu|=j\\\mu\geq
  \rho-\eta_i}}\cfrac{|a_{i,\mu-\rho+\eta_i}|\mu!}{(\mu-\rho+\eta_i)!\rho!}
\\
  &\leq C_a\sum_{\ell=1}^{j-1}\sum_{|\rho|=\ell}|u_{\rho}|\sum_{\substack{|\mu|=j\\\mu\geq
  \rho-\eta_i}}r^{j-\ell}\cfrac{\mu!}{\rho!}\,\cfrac{1}{(1+j-\ell)^m}.
\end{align*}
Again we have loosened $(r/4)^j$ into $r^j$.

Letting $\rho_M$ denote again the large index of $\rho$, and $\rho_m$
its smallest non-zero index, then
\[
  \max_{\substack{|\mu|=j\\\mu\geq \rho-\eta_i}}\cfrac{\mu!}{\rho!}=\cfrac{(\rho_M+j-\ell+1)!}{\rho_M!\rho_m}\leq
  \cfrac{j!}{(\ell-1)!}\leq \frac{(j+1)!}{\ell!}.
\]
In particular, by the induction hypothesis \eqref{eq:ctrl_u_rec_3.5},
\begin{align*}
  \left|\sum_{|\nu|\geq
        2}\cfrac{a_{i,\nu}u_{\mu-\nu+\eta_i}\mu!}{\nu!(\mu-
      \nu+\eta_i)!}\right|
  &\leq
    C_aC(h,f',\varphi)C_gr^j(j+1)!k!\sum_{\ell=1}^{j-1}\,\cfrac{(1+j-\ell)^d}{(1+j-\ell)^m(1+k+\ell)^m}\\
  &\leq
    C_aC(h,f',\varphi)C_g\cfrac{r^j(j+1)!k!}{(1+j+k)^m}\sum_{\ell=1}^{j-1}\cfrac{(1+j-\ell)^d(1+j+k)^m}{(1+j-\ell)^m(1+k+\ell)^m}\\
  &\leq C_aC(d)\cfrac{3^m}{4^m}\,C(h,f',\varphi)C_g\cfrac{r^j(j+1)!k!}{(1+j+k)^m}.
\end{align*}
Hence, for $m$ large enough, and $r$ large enough accordingly, one has,
for every $1\leq i\leq d$,
\[
  \left|\sum_{|\nu|\geq
        2}\cfrac{a_{i,\nu}u_{\mu-\nu+\eta_i}\mu!}{\nu!(\mu-
      \nu+\eta_i)!}\right|\leq \frac {1}{3dC_{\omega}}
  C(h,f',\varphi)C_g\cfrac{r^j(j+1)!k!}{(1+k+j)^m}.
  \]

To conclude, if $C(h,f',\varphi)\geq 3C_{\omega}$, then
\[
  \sum_{|\mu|=j}|u_{\mu}|\leq\frac{1}{j+1}\left(\frac 13
    C(h,f',\varphi)+\frac 13 C(h,f',\varphi)+\frac{1}{3}
    C(h,f',\varphi)\right)C_g\cfrac{r^j(j+1)!k!}{(1+k+j)^m},
\]
which concludes the induction.

{\bf Third step}

Let $U$ be a neighbourhood of $0$ such that all trajectories of $X$,
starting in $U$, converge to $0$ (exponentially fast) in negative time.
It remains to prove that $u$ is well-defined and holomorphic on $U$. Since the
sequence of derivatives of $u$ at $0$ enjoys an analytic-type growth control, the
associated power series converges on some small neighbourhood $W$ of
$0$. Then, from the knowledge of $u$ on $W$ one can build $u$ on $U$
using the geometric structure of the transport equation. Indeed, by
definition $0$
is the repulsive point of all trajectories of $X$ on $U$.
Letting $(\Phi_t)_{t\in \R}$ denote the flow of $-X$, there exists $T>0$ such that $\Phi_T(U)\subset
W$. Then the transport equation on $u$ implies the Duhamel formula
\[
  u(x)=u(\Phi_T(x))+\int_0^Tg(\Phi_t(x))\dd t +
  \int_0^Tu(\Phi_t(x))h(\Phi_t(x))\dd t.\]
By the analytic Picard-Lindel\"of theorem, the unique solution of this degree $1$ integral equation, where the
initial data $u(\Phi_T(x))$ and the coefficients have real-analytic dependence
on $\Phi_T(x)\in W$, is well-defined and real-analytic. Then $u$ is
well-defined on $U$, and holomorphic since the derived equation on
$\overline{\partial}u$ is $\overline{\partial}u=0$.

{\bfseries Fourth step}

Now we impose the stronger control \eqref{eq:ctrl_g_sharp_3.5} on $g$
and prove \eqref{eq:ctrl_u_sharp_3.5}.
Observe that, if $j\geq k$ and
\[
  \sum_{|\mu|=j}|g_{\mu}|\leq
  C_g\frac{(r/4)^j(j+k+1)!}{(1+j+k+1)^{m-\frac 12}},
\]
and if $C(h,f,\varphi)\geq 6C_{\omega}$, then
\[
  \sum_{|\mu|=j}\frac{|g_{\mu}|}{\left|\sum_{i=1}^d\omega_i\mu_i\right|}\leq
  \frac 16 C(h,f,\varphi)C_g\frac{(r/4)^j(j+k+1)!}{(1+k+j+1)^{m-\frac 12}(j+1)}\leq
  \frac 13 C(h,f,\varphi)C_g\frac{(r/4)^j(j+k)!}{(1+k+j)^{m-\frac 12}}.
\]
It then remains to study how the more precise condition on $u$
propagates. Fix $j\geq k$; suppose that \eqref{eq:ctrl_u_3.5} is satisfied for all
$\ell<j$, and that \eqref{eq:ctrl_u_sharp_3.5} is satisfied for all
$k\leq \ell<j$. Then
\[
  \sum_{|\mu|=j}\sum_{|\nu|\geq
    1}\left|\frac{h_{\nu}u_{\mu-\nu}\mu!}{\nu!(\mu-\nu)!}\right|=\sum_{|\mu|=j}\sum_{1\leq
    |\nu|\leq
    j-k-1}\left|\frac{h_{\nu}u_{\mu-\nu}\mu!}{\nu!(\mu-\nu)!}\right|+\sum_{|\mu|=j}\sum_{|\nu|\geq
    j-k}\left|\frac{h_{\nu}u_{\mu-\nu}\mu!}{\nu!(\mu-\nu)!}\right|.
\]
In the first sum, one has $|\mu-\nu|\geq k$. Hence
\[
  \sum_{|\mu|=j}\sum_{1\leq
    |\nu|\leq
    j-k-1}\left|\frac{h_{\nu}u_{\mu-\nu}\mu!}{\nu!(\mu-\nu)!}\right|=\sum_{\ell=k}^{j-1}\sum_{|\rho|=\ell}|u_{\rho}|\sum_{\substack{|\mu|=j\\\mu\geq
      \rho}}\frac{|h_{\mu-\rho}|\mu!}{\nu!(\mu-\nu)!}.
\]
From there and \eqref{eq:ctrl_h_gen}, one has, as previously,
\begin{align*}
  \sum_{|\mu|=j}\left|\sum_{1\leq
    |\nu|\leq
    j-k-1}\cfrac{h_{\nu}u_{\mu-\nu}\mu!}{\nu!(\mu-\nu)!}\right|&\leq\sum_{\ell=k}^{j-1}\sum_{|\rho|=\ell}|u_{\rho}|\sum_{\substack{|\mu|=j\\\mu\geq\rho}}\cfrac{|h_{\mu-\rho}|}{(\mu-\rho)!}\cfrac{\mu!}{\rho!}\\
  &\leq
    C_h\sum_{\ell=k}^{j-1}(r/4)^{j-\ell}\sum_{|\rho|=\ell}|u_{\rho}|\sum_{\substack{|\mu|=j\\\mu\geq\rho}}\cfrac{\mu!}{\rho!}\cfrac{1}{(1+j-\ell)^m}\\
  &\leq
    C_h\sum_{\ell=k}^{j-1}(r/4)^{j-\ell}\frac{j!}{\ell!}\sum_{|\rho|=\ell}|u_{\rho}|\sum_{\substack{|\mu|=j\\\mu\geq\rho}}\cfrac{1}{(1+j-\ell)^{m-\frac
  12}}\\
  &\leq
    C_hC(h,f',\varphi)C_g\frac{(r/4)^j(j+k)!}{(1+j+k)^{m-\frac 12}}\sum_{\ell=1}^{k}\underbrace{\frac{j!(\ell+k)!}{\ell!(j+k)!}}_{\leq
    1}\cfrac{(1+j-\ell)^d(1+j+k)^{m-\frac 12}}{(1+j-\ell)^{m-\frac 12}(1+\ell+k)^{m-\frac 12}}.
\end{align*}
If $m$ is large enough, and $r$ is large accordingly,
we obtain
\[
  \sum_{|\mu|=j}\left|\sum_{1\leq
    |\nu|\leq
    j-k-1}\cfrac{h_{\nu}u_{\mu-\nu}\mu!}{\nu!(\mu-\nu)!}\right|\leq
\frac{1}{6C_{\omega}}C(h,f',\varphi)C_g\frac{(r/4)^j(j+k)!}{(1+k+j)^{m+\frac
  12}}.
\]

In the second sum, we have
\begin{align*}
\sum_{\ell=1}^k\sum_{|\rho|=\ell}|u_{\rho}|\sum_{\substack{|\mu|=j\\
    \mu\geq \rho}}\frac{|h_{\mu-\rho}|\mu!}{(\mu-\rho)!\rho!} &
\leq
C_hC(h,f',\varphi)C_g\sum_{\ell=1}^{j}(r/4)^{j-\ell}\frac{j!(1+j-\ell)^d}{\ell!(1+j-\ell)^m}\sum_{|\rho|=\ell}|u_{\rho}|\\
& \leq
                                                                                                     C_hC(h,f',\varphi)C_gr^jj!k!4^{k-j}\sum_{\ell=1}^k\frac{(1+j-\ell)^d}{(1+j-\ell)^m(1+k+\ell)^m}.
\end{align*}
Let us prove that, since $k\leq j$, one has
\[
  \frac{4^{k}j!k!}{(j+k)!\sqrt{j+k+1}}\leq 2.
\]
This is a log-convex function of $k$; at $k=0$ it is equal to
$1/\sqrt{j+1}$. at $k=j$ we use the fact that
\[
  4^jj!j!\leq 4^j(2j)!\times \frac{j!j!}{(2j)!}\leq 2\sqrt{2j+1}(2j)!,
\]
as remarked before the proof.

In particular, since $\sqrt{j+k+1}\leq \sqrt{(j-\ell+1)(k+\ell+1)}$,
one has
\[
  \sum_{\ell=1}^k\sum_{|\rho|=\ell}|u_{\rho}|\sum_{\substack{|\mu|=j\\
      \mu\geq \rho}}\frac{|h_{\mu-\rho}|\mu!}{(\mu-\rho)!\rho!}\leq
  2C_hC(h,f',\varphi)C_g\frac{(r/4)^j(j+k)!}{(1+j+k)^{m-\frac
      12}}\sum_{\ell=1}^k\frac{(1+j-\ell)^d(1+j+k)^{m-\frac
      12}}{(1+j-\ell)^{m-\frac 12}(1+k+\ell)^{m-\frac 12}}.
  \]

We finally obtain, for $m$ large enough, and $r\geq 2^{2+m-m_0}$,
\[
  \sum_{|\mu|=j}\left|\sum_{|\nu|\geq
      1}\frac{h_{\nu}u_{\mu-\nu}\mu!}{\nu!(\mu-\nu)!}\right|\leq
  \frac{1}{3C_{\omega}}C(h,f',\varphi)C_g\frac{(r/4)^j(j+k)!}{(1+j+k)^{m-\frac
    12}}.
\]
The control on
\[
  \left|\sum_{i=1}^d\sum_{|\nu|\geq
      2}\frac{a_{i,\nu}u_{\mu-\nu+\eta_i}\mu!}{\nu!(\mu-\nu+\eta_i)!}\right|
\]
is very similar; the only notable difference is the combinatorial
factor studied in Part 2,
\[
  \max_{\substack{|\mu|=j\\\mu\geq \rho-\eta_i}}\frac{\mu!}{\rho!}\leq
  \frac{(j+1)!}{\ell!}=(j+1)\frac{j!}{\ell!},\]
which brings a supplementary factor $j+1$ in all cases. We obtain
\[
  \left|\sum_{|\mu|=j}\sum_{i=1}^d\sum_{|\nu|\geq
      2}\frac{a_{i,\nu}u_{\mu-\nu+\eta_i}\mu!}{\nu!(\mu-\nu+\eta_i)!}\right|\leq
  \frac{1}{3C_{\omega}}C(h,f',\varphi)C_g\frac{(r/4)^j(j+k)!}{(1+j+k)^{m-\frac
      12}}(j+1),
\]
and finally,
\[
\left|\sum_{|\mu|=j}\frac{1}{\sum_{i=1}^d\omega_i\mu_i}\left[\sum_{|\nu|\geq
      1}\cfrac{h_{\nu}u_{\mu-\nu}\mu!}{\nu!(\mu-\nu)!}+g_{\mu}-\sum_{i=1}^d\sum_{|\nu|\geq
      2}\cfrac{a_{i,\nu}u_{\mu-\nu+\eta_i}\mu!}{\nu!(\mu-\nu+\eta_i)!}\right]\right|\leq
C(h,f',\varphi)C_g\frac{(r/4)^j(j+k)!}{(1+j+k)^{m-\frac 12}},
\]
which concludes the proof.
\end{proof}

\section{Construction of quasimodes}
\label{sec:almost-eigenvectors}

Solving the Hamilton-Jacobi equation then controlling successive transport equations
allows us to prove the first part of Theorem A, which is the object of
this section.

The strategy of proof is the following: we first exhibit sequences
$(u_i)_{i\geq 0}$ and $(\lambda_i)_{i\geq 0}$ such that the eigenvalue
equation \eqref{eq:eigval}
is valid up to $O(N^{-\infty})$, and we control these sequences in
analytic spaces. Then we prove that one can perform an analytic summation
in \eqref{eq:eigval}.

Before proceeding, we note that, if $\varphi$ is admissible and $u(N)$ is the summation of an
  analytic symbol, both being defined on an open neighbourhood $V$ of
  $0$, then $\mathds{1}_Ve^{N\varphi}u(N)\psi_{0}^N$ concentrates at
  $0$, in the sense that there exist $C>0,c>0$ such that for every open set $W\subset M$,
  \[
    \|N^{-\frac d2}\mathds{1}_Ve^{N\varphi}u(N)\psi_0^N\|_{L^2(W)}\leq Ce^{-cN\dist(W,\{0\})^2}.
    \]
    and moreover, by Proposition \ref{prop:Szeg-gen} and
  the stationary phase lemma, there exists $C>0$ such that, for every
  $N\in \N$, there holds
  \[
    \frac 1C N^{\frac d2}\leq
    \|\mathds{1}_Ve^{N\varphi}u(N)\psi_{0}^N\|_{L^2(M)}\leq CN^{\frac d2}.
  \]
  In particular, if
  \[\|(T_N(f)-\lambda(N))\mathds{1}_Ve^{N\varphi}u(N)\psi_{0}^N\|_{L^2(M)}\leq
  Ce^{-c'N}\] then $\lambda(N)$ will be exponentially close to the
  spectrum of $T_N(f)$. Thus, through Proposition \ref{prop:sum-WKB} we are indeed
  providing quasimodes of $T_N(f)$ which concentrate on $0$.

\begin{prop}
  \label{prop:WKB-formal}
  Let $\varphi$ denote an admissible solution to the Hamilton-Jacobi
  equations \eqref{eq:Hamilton-Jacobi}, and let $\psi_{0}^N$ denote
  the sequence of coherent states at $0$. There exists
  $W\subset \subset V\subset\subset U\subset U_0$ containing zero, a sequence
  $(u_k)_{k\geq 0}$ of holomorphic functions on $U$, and a sequence
  $(\lambda_k)_{k\geq 0}$ of real numbers, such that for every $K\geq
  0$ there holds
  \[
    \left\|\left(T_N(f)-\sum_{k=0}^KN^{-k-1}\lambda_k\right)\mathds{1}_V\psi_0^Ne^{N\varphi}\sum_{k=0}^KN^{-k}u_k\right\|_{L^2(W)}=O(N^{-\frac{d}{2}-K-2}).
    \]
    One has
    \[
      \lambda_0=\min\Sp(T_1(\Hess(f)(0))).
    \]
  \end{prop}
  \begin{proof}
    Recall that, by Proposition \ref{prop:Szeg-gen}, there exists an
    analytic symbol $a$ and constants $c>0, c'>0$ such that 
    \[
      S_N(x,y)=N^d\Psi^{\otimes
        N}(x,\overline{y})\sum_{k=0}^{cN}N^{-k}a_k(x,\overline{y})+O(e^{-c'N}).
    \]
    In particular,
    \[
      \psi_0^N(x)=N^d\Psi^{\otimes N}(x,0)\sum_{k=0}^{cN}N^{-k}a_k(x,0)+O(e^{-c'N}).
    \]
    Let $(u_k)_{k\in \N}$ be a sequence of holomorphic functions on
    $U$ and let
    \[
      u(N)=\sum_{k=0}^KN^{-k}u_k.
      \]
    With $a(N)=\sum_{k=0}^{cN}a_k$, by definition of $\Phi_1$,
    one has, uniformly for $x\in W$,
\begin{multline*}
T_N(f)\left(\mathds{1}_V\psi_{0}^Ne^{N\varphi}u(N)\right)(x)\\=\psi_{0}^N(x)e^{N\varphi(x)}\int_{y\in V}
e^{N(\Phi_1(x,y,\overline{y},0)+\varphi(y)-\varphi(x))}\frac{a(N)(x,\overline{y})}{a(N)(x,0)}a(N)(y,0)\widetilde{f}(y,\overline{y})u(N)(y)\dd
y+O(e^{-c'N}).
\end{multline*}
We are now able to apply the complex stationary phase Lemma (with
analytic phase but, at this stage, smooth symbol, as in \cite{melin_fourier_1975}). Let $*$ denote the Cauchy product of symbols, and let $b$ be the analytic symbol such that
\[
b(x,y,\overline{w})=\widetilde{f}(y,\overline{w})a(x,\overline{w})*a^{*-1}(x,0)*a(y,0)J(x,y,\overline{w}),
\]
  where $J$ is the Jacobian of the change of variables $\kappa_x$
  defined in \eqref{eq:defkappa}. One has
\begin{equation}
  \label{eq:Conj-TN-Phase}
  e^{-N\varphi(x)}T_N(f)\left(\psi_{0}^Ne^{N\varphi}u\right)(x)\\=\psi_{0}^N(x)\sum_{k=0}^{+\infty}N^{-k}\sum_{n=0}^k\left.\cfrac{\Delta_{\kappa_x}^n}{n!}(u(y)b_{k-n}(x,y,\overline{w}))\right|_{(y,\overline{w})=(x,\overline{y_c}(x))}+O(N^{-\infty}).
\end{equation}
Using Proposition \ref{prop:transport} with
\[
  f':(x,y,\overline{w})\mapsto b_0(x,y,\overline{w}),
  \] which indeed coincides with $f$ up to $O(|x,y,\overline{w}|^3)$, we will construct by induction a
sequence of holomorphic functions $u_i$ and a sequence of real numbers
$\lambda_i$ such that
\begin{equation}
  \label{eq:eigval}
  T_N(f)\left(\psi_{0}^Ne^{N\varphi}\sum_{k=0}^{+\infty}N^{-k}u_k\right)(x)=\psi_{0}^N(x)e^{N\varphi(x)}\left(\sum_{j=0}^{+\infty}N^{-j-1}\lambda_j\right)\left(\sum_{k=0}^{cN}N^{-k}u_k(x)\right)+O(N^{-\infty}).
\end{equation}
We further require that
\[
  u_k(0)=\begin{cases}1&\text{if }k=0\\
    0&\text{else.}
  \end{cases}
\]

In the right-hand side of \eqref{eq:eigval}, there are no terms of
order $0$. In the left-hand side, the term of degree $0$ is given by
the term $k=0$ in \eqref{eq:Conj-TN-Phase}, so that one needs to solve
\[
  \widetilde{f}(x,\overline{y}_c(x))u_0(x)\cfrac{a_0(x,0)}{a_0(y,0)}\,a_0(x,\overline{y})J(x,x,\overline{y}_c(x))=b_0(x,x,\overline{y}_c(x))u_0(x)=0.
\]

Since $\widetilde{f}(x,\overline{y}_c(x))=0$, this equation is always satisfied.

By the stationary phase lemma \eqref{eq:Conj-TN-Phase}, the order $1$ in \eqref{eq:eigval} reads
\begin{equation}
  \label{eq:u0}
  \lambda_0u_0(x)-(\Delta_{\kappa_x}b_0)(x,x,\overline{y}_c(x))u_0(x)-(\nabla_{\kappa_x}b_0)(x,x,\overline{y}_c(x))\cdot
  \nabla_{\kappa_x}u_0(x)=0.
\end{equation}
 Here, and until the end of this proof as well as that of
 Proposition~\ref{prop:Wells-WKB}, we (informally) denote
 \[
   \nabla_{\kappa_x}u_k(x)=\nabla_{\kappa_x}[(x,y,\overline{w})\mapsto
   u_k(y)]_{(y,\overline{w})=(x,\overline{y}_c(x))}.\]
  
The equation \eqref{eq:u0} allows us to solve for $u_0$ with the
supplementary condition $u_0(0)=1$. Indeed, as
$\nabla_{\kappa_x}b_0(0)=0$, at $x=0$, the order $1$ reads 
\[
  \lambda_0-(\Delta_{\kappa_x}b_0)(0,0,0)=0,
\]
so that we set
\[
  \lambda_0=(\Delta_{\kappa_x}b_0)(0,0,0).
\]

We now prove that $\lambda_0$  coincides with the ground state energy of the
associated quadratic operator $T_N(\Hess(f)(0))$. Indeed, $\lambda_0$ depends only on the Hessian
of $f$ and $\phi$ at zero (which together determine the Hessian of
$\varphi$ at zero as seen in Proposition
\ref{prop:Hamilton-Jacobi-sol}, thus they determine the linear part of
the change of variables $\kappa_0$, which in turn
determines $\Delta_{\kappa_0}$ and $J$ at $0$). If $f$ and $\phi$ are quadratic, then the solution $\varphi$ of
the Hamilton-Jacobi equation is also quadratic as constructed in Proposition \ref{prop:Hamilton-Jacobi-sol}, so that $u_0=1$
satisfies \eqref{eq:eigval} exactly. Thus, $\lambda_0$ is an eigenvalue
of $T_N(\Hess(f)(0))$ which depends continuously on
$\Hess(f)(0)$. Moreover, if
$\Hess(f)(0):y\mapsto |y|^2$, then $\Hess(\varphi)=0$ so that the
eigenvector of $T_N(\Hess(f)(0))$ associated with $\lambda_0$ is the
coherent state (in $\C^d$) $\psi^N_0$, which is the ground state of
$T_N(|y|^2)$; thus in this case $\lambda_0$ is the ground state
energy. Since the set of positive definite quadratic forms in
$\R^{2d}$ is connected, and since there is always a gap between the
ground state energy and the first excited level, then $\lambda_0$ is
always the ground state energy of $T_N(\Hess(f)(0))$.

We wish now to find $u_0$ such that $u_0(0)=1$. Setting $v_0=u_0-1$ yields
\[
  \nabla_{\kappa_x}v_0(x)\cdot
  (\nabla_{\kappa_x}b_0)(x,x,\overline{y}_c(x))=v_0(x)\left[(\Delta_{\kappa_x}b_0)(x,x,\overline{y}_c(x))-(\Delta_{\kappa_x}b_0)(0,0,0)\right].
\]
We then solve for $v_0$ using Proposition \ref{prop:transport} with $f'=b_0$, which
indeed yields $v_0(0)=0$.

Let us now find the remaining terms of the sequences $(u_k)_{k\geq 0}$
and $(\lambda_k)_{k\geq 0}$ by induction. For $k\geq 1$, the term of order $k+1$ in \eqref{eq:eigval} is given
again by the stationary phase lemma \eqref{eq:Conj-TN-Phase}: at this order, the equation is
  \begin{multline}
    \label{eq:Induction-uk}
    \lambda_ku_0(x)+\lambda_0u_k(x)-(\Delta_{\kappa_x}b_0)(x,x,\overline{y}_c(x))u_k(x)-
    (\nabla_{\kappa_x}b_0)(x,x,\overline{y}_c(x))\cdot
    \nabla_{\kappa_x}u_k(x)\\=-\sum_{j=1}^{k-1}\lambda_ju_{k-j}(x)
    +\sum_{n=2}^{k+1}\sum_{l=0}^{k+1-n}\left.\cfrac{\Delta_{\kappa_x}^{n}}{n!}\left(u_{l}(y)b_{k+1-n-l}(x,y,\overline{w})\right)\right|_{(y,\overline{w})=(x,\overline{y_c}(x))}.
  \end{multline}
  In this equation, we have put to the left-hand side all terms
  involving $\lambda_k$ or $u_k$, and all terms involving $\lambda_l$
  and $u_l$ with $l<k$ to the
  right-hand side. We can apply Proposition
  \ref{prop:transport} to solve for $u_k,\lambda_k$ once
  $(u_l,\lambda_l)_{0\leq l\leq k-1}$ are known. Indeed,
\eqref{eq:Induction-uk} takes the form
\begin{equation}\label{eq:u_given_by_g}
  (\nabla_{\kappa_x}b_0)(x,x,\overline{y}_c(x))\cdot
  \nabla_{\kappa_x}u_k(x)=g_k(x)+h(x)u_k(x),
\end{equation}
with $h(x)=\Delta_{\kappa_x}b_0(x,x,\overline{y}_c(x))-\lambda_0$
and
\begin{equation}\label{eq:g_given_by_u}
  g_k(x)=-\sum_{l=1}^{k-1}\lambda_lu_{k-l}(x)-\lambda_ku_0+
  \sum_{n=2}^{k+1}\sum_{l=0}^{k+1-n}\left.\frac{\Delta^n_{\kappa_x}}{n!}(u_l(y)b_{k+1-n-l}(x,y,\overline{w}))\right|_{(y,\overline{w})=(x,\overline{y}_c(x))}.
\end{equation}
By construction of $\lambda_0$, one has $h(0)=0$; moreover,
\[
  g_k(0)=\sum_{n=2}^{k+1}\sum_{l=0}^{k+1-n}\left.\frac{\Delta^n_{\kappa_0}}{n!}(u_l(y)b_{k+1-n-l}(0,y,\overline{w}))\right|_{(y,\overline{w})=(0,0)}-\lambda_k.
\]
Thus, one can solve for $\lambda_k$ by setting $g_k(0)=0$, then solve
for $u_k$ using Proposition \ref{prop:transport}: the role of $f'$ is
played by $b_0$, which does not depend on $k$. Thus, letting $U$ be as
in Proposition \ref{prop:transport}, one can, by induction on $k$,
define $g_k$ as a holomorphic function on $U$ using
\eqref{eq:g_given_by_u}, then $u_k$ as a holomorphic function on $U$
using \eqref{eq:u_given_by_g}.
\end{proof}

It remains to prove that, because of Proposition \ref{prop:transport},
the coefficients $(u_k)_{k\geq 0}$ and $(\lambda_k)_{k\geq 0}$ satisfy
analytic growth controls.

\begin{prop}\label{prop:Wells-WKB}
  Let $(u_k)_{k\geq 0}$ and $(\lambda_k)_{k\geq 0}$ be the sequences
  constructed in the previous proposition. Then there exist $C>0$, $R>0$,
  $r>0$, $m\in \R$ and an open set $V\subset\subset U$ containing $0$
  such that, for all $k\geq 0,j\geq 0$, one has
  \begin{align*}
    \|u_k\|_{C^j(V)}&\leq C\frac{r^jR^kj!k!}{(j+k+1)^m}\\
    |\lambda_k|&\leq C\frac{R^kk!}{(k+2)^m}.
  \end{align*}
  Moreover, if $j\geq k$, then
  \[
    \|u_k\|_{C^j(V)}\leq C\frac{(r/4)^jR^k(j+k)!}{(j+k+1)^{m-\frac 12}}.
    \]
\end{prop}

\begin{proof}
The proof proceeds by induction on $k$ and consists in
three steps. In the first step, we show that in equation
\eqref{eq:Induction-uk} (that is, in the definition of $g_k$), when
expanding $\Delta^n_{\kappa_x}(u_lb_{k+1-n-l})$, no derivatives of $u_l$ of order larger than
$n$ appear. This will allow us to apply Lemma
\ref{lem:new_lemma_4.6}. The second step is the core of the induction:
we suppose some control
on all derivatives of $u_l$ at zero, for $0\leq l\leq k-1$, and we
apply Lemma \ref{lem:new_lemma_4.6} to deduce that the derivatives of
$g_k$ at zero are well-behaved. We then apply Proposition
\ref{prop:transport} to obtain a control on the derivatives of
$u_{k}$ at zero. In the last step, we deduce, from a control of the
derivatives of $u_k$ at zero, a control of the same nature on a small
open neighbourhood.

{\bfseries First step.}

Let $f_0$ be a holomorphic function near $0$ in $M$. Then $T_N(f_0)$ is,
locally, a multiplication operator, so that, for all holomorphic $u$, \[e^{-N\varphi}T_N(f_0)(\psi_{0}^Ne^{N\varphi}u)=\psi_{0}^Nf_0u+O(e^{-c'N}).\]
In this particular case, no derivative of $u$ of order $\geq
1$ appear in \eqref{eq:Conj-TN-Phase}, hence in \eqref{eq:Induction-uk}.

We then
decompose the real-analytic function $f$ as \[\widetilde{f}(y,\overline{y})=\widetilde{f}(y,\overline{y}_c(x))+\left(\widetilde{f}(y,\overline{y})-\widetilde{f}(y,\overline{y}_c(x))\right).\]
In the right-hand side, the
second term vanishes when $\overline{y}=\overline{y}_c(x)$, so that, with
\[
  \Phi:(x,y,\overline{w})\mapsto \Phi_1(x,y,\overline{w},0)+\varphi(y)-\varphi(x),
\]
there exists a smooth vector-valued function $f_1$ such that
\[
  \widetilde{f}(y,\overline{y})=\widetilde{f}(y,\overline{y}_c(x))+\partial_y\Phi(x,y,\overline{y})\cdot
  f_1(x,y,\overline{y}).
\]
Now $S_N$ acts as the identity on holomorphic functions and
$\overline{y}_c$ is a holomorphic function of $x$ so that, by
integration by parts:
\begin{multline*}
  \int
  e^{-N\Phi(x,y,\overline{y})}a(N)(x,\overline{y})\widetilde{f}(y,\overline{y})u(y)\dd
  y\\ = \psi_0^N(x)\widetilde{f}(x,\overline{y}_c(x))u(x)+\int
  e^{-N\Phi(x,y,\overline{y})}a(N)(x,\overline{y})\partial_y\Phi(x,y,\overline{y})\cdot
  f_1(x,y,\overline{y})\,u(y)\dd
  y+O(e^{-c'N})\\
  =\psi_0^N(x)\widetilde{f}(x,\overline{y}_c(x))u(x)+N^{-1}\int
  e^{-N\Phi(x,y,\overline{y})}a(N)(x,\overline{y})\partial_y\left[f_1(x,y,\overline{y})u(y)\right]\dd y+O(e^{-c'N}).
\end{multline*}
In particular, in the term of order $N^{-1}$ in
\eqref{eq:Conj-TN-Phase}, there only are derivatives of $u$ of order
$0$ or $1$.

One can in fact perform this decomposition iteratively: with
\[
  \partial_y\left[f_1(x,y,\overline{y})u(y)\right]=\partial_y\cdot
  f_1(x,y,\overline{y})u(y)+f_1(x,y,\overline{y})\partial u(y),
\]
one can write
\begin{align*}
  f_1(x,y,\overline{y})&=f_1(x,y,\overline{y}_c(x))+\partial_y\Phi(x,y,\overline{y})\cdot
                         f_{2,0}(x,y,\overline{y})\\
  \partial_yf_1(x,y,\overline{y})&=\partial_yf_1(x,y,\overline{y}_c(x))+\partial_y\Phi(x,y,\overline{y})\cdot
                         f_{2,1}(x,y,\overline{y}),
\end{align*}
so that the original integral is equal to
\begin{multline*}
  \psi_0^N(x)\widetilde{f}(x,\overline{y}_c(x))u(x)\\+N^{-1}\psi_0^N(x)\left[f_1(x,x,\overline{y}_c(x))\partial
    u(x)+\partial_yf_1(x,x,\overline{y}_c(x))u(x)\right]\\+N^{-2}\int
  e^{-N\Phi(x,y,\overline{y})}a(N)(x,\overline{y})\partial_y[f_{2,0}(x,y,\overline{y})\partial
  u(y)+f_{2,1}(x,y,\overline{y})u(y)]\dd y \\+ O(e^{-c'N}).
\end{multline*}

  By induction, the terms of order $N^{-k}$ in the expansion
  \eqref{eq:Conj-TN-Phase} only contain derivatives of $u$ of order
  smaller than $k$. This means in particular that, in
  \eqref{eq:Induction-uk}, in
  \[
\left.
  \Delta^n_{\kappa_x}(u_l(y)b_{k+1-n-l}(x,y,\overline{w}))\right|_{(y,\overline{w})=(x,\overline{y}_c(x))},
\]
there only appears derivatives of $u_l$ of order less or equal to $n$.


  {\bfseries Second step.}

  Let us prove by induction that the sequences $(u_k)_{k\geq 0}$ and
  $(\lambda_k)_{k\geq 0}$ are analytic symbols. We will make use of the precise controls obtained in Proposition
  \ref{prop:transport}.
  Since $(b_k)_{k\geq 0}$ is an analytic symbol and $u_0$ is holomorphic, by
  Proposition \ref{prop:anal-symb} there exists a small open
  neighbourhood $W$ of zero in $\C^d$, and a small open neighbourhood
  $W_1$ of $0$ in $\C^{3d}$, and
  $r_0,R_0,m_0,C_b,C_0>0$ such that
  \begin{align*}
    \|(x,v_1,v_2)\mapsto b_k(x,\kappa^{-1}_x(v_1,v_2))\|_{C^j(W_1)}&\leq C_b\cfrac{r_0^jR_0^k(j+k)!}{(j+k+1)^{m_0}}\\
    \|u_0\|_{C^j(W)}&\leq C_0\cfrac{r_0^jj!}{(j+1)^{m_0}}\\
    \|\kappa^{-1}\|_{C^j(W_1)}&\leq C_{\kappa}\frac{r_0^jj!}{(j+1)^{m_0}}.
  \end{align*}
Here, and the rest of this proof we again denote by $\kappa^{-1}$ the map $(x,v_1,v_2)\mapsto
  (x,\kappa^{-1}_x(v_1,v_2))$.
  
  Let us transform this into a control on $b_k$ which is more suited
  to our needs. First, for all $j$ and $k$, one has
  \[
    \|b_k\circ \kappa^{-1}\|_{C^j(W_1)}\leq C_b\cfrac{(4r_0)^j(4R_0)^kj!k!}{(j+k+1)^{m_0+1}}.
  \]
  Indeed $(j+k)!\leq 2^{j+k}j!k!$ and $2^{j+k}\geq j+k+1$. In
  particular,
  \[
    \|b_0\circ \kappa^{-1}\|_{C^j(W_1)}\leq C_b\cfrac{(4r_0)^jj!}{(j+1)^{m_0+1}}\leq C_b\cfrac{(4r_0)^jj!}{(j+1)^{m_0}}.
  \]
  On the other hand, for $k\geq 1$, one has
  \[
    \|b_k\circ \kappa^{-1}\|_{C^j(W_1)}\leq
    C_b\cfrac{(4r_0)^j(4R_0)^kj!(k-1)!}{(j+k+1)^{m_0}},
  \]
  since $\frac{k}{j+k+1}\leq 1$.

  In particular, for any $m\geq m_0$, for any $r\geq 2^{m+5-m_0}r_0$ and
  $R\geq 2^{m+2-m_0}R_0$, one has
  \begin{align*}
    \|b_0\circ \kappa^{-1}\|_{C^j(W_1)}&\leq C_b\cfrac{(r/8)^jj!}{(j+1)^m}\\
    \|b_k\circ \kappa^{-1}\|_{C^j(W_1)}&\leq
                        C_b\cfrac{(r/8)^jR^kj!(k-1)!}{(j+k+1)^m}\qquad
    \qquad k\geq 1\\
    \|u_0\|_{C^j(W)}&\leq C_0\cfrac{(r/32)^jj!}{(j+1)^{m}}\leq
                      C_0\frac{(r/4)^jj!}{(j+1)^{m-\frac 12}}\\
    \|\kappa^{-1}\|_{C^j(W_1)}&\leq
    C_{\kappa}\frac{(r/16)^{j}(j-1)!}{j^m}\qquad \qquad j\geq 1.
  \end{align*}
  
  In equation \eqref{eq:Induction-uk}, let us isolate the terms
  involving $u_0$. We obtain
  \begin{multline}\label{eq:rec_nice_form}
    \lambda_ku_0(x)+\lambda_0u_k(x)-\Delta_{\kappa_x}b_0(x,x,\overline{y}_c(x))u_k(x)-
    \nabla_{\kappa_x}b_0(x,x,\overline{y}_c(x))\cdot
    \nabla_{\kappa_x}u_k(x)\\=
    \sum_{n=2}^{k+1}\left.\cfrac{\Delta^n_{\kappa_x}}{n!}\left(u_{0}(y)b_{k+1-n}(x,y,\overline{w})\right)\right|_{(y,\overline{w})=(x,\overline{y}_c(x))}\\-\sum_{j=1}^{k-1}\lambda_ju_{k-j}(x)+\sum_{n=2}^{k+1}\sum_{l=1}^{k+1-n}\left.\cfrac{\Delta_{\kappa_x}^{n}}{n!}\left(u_{l}(y)b_{k+1-n-l}(x,y,\overline{w})\right)\right|_{(y,\overline{w})=(x,\overline{y}_c(x))}.
  \end{multline}
  Let $m,r,R,C_u,C_{\lambda}$ be large enough (they will be fixed
  in the course of the induction), and suppose that, for all $0\leq l\leq k-1$ and all
  $j\geq 0$, one has
  \begin{align}\label{eq:ctrl_lambda_4.2}
    |\lambda_l|&\leq C_{\lambda}\cfrac{R^ll!}{(l+2)^m}\\\label{eq:ctrl_u_4.2}
    \|\nabla^ju_l(0)\|_{\ell^1}&\leq C_u\cfrac{r^jR^lj!l!}{(j+l+1)^m}.
  \end{align}
  Suppose further that for $j\geq l$ one has the more precise control
  \begin{equation}\label{eq:ctrl_u_sharp_4.2}
    \|\nabla^ju_l(0)\|_{\ell^1}\leq
    C_u\cfrac{(r/4)^jR^l(j+l)!}{(j+l+1)^{m-\frac 12}}.
  \end{equation}
  Our goal is now to prove the three inequalities
  \eqref{eq:ctrl_lambda_4.2}, \eqref{eq:ctrl_u_4.2}, and \eqref{eq:ctrl_u_sharp_4.2},  in the case
  $l=k$.
  
  To begin with,
  we estimate how the iterated modified Laplace operator
  $\Delta_{\kappa_x}^n$ acts on $u_{l}$ using the fact that the former
  differentiates the latter at most $n$ times (Part 1) and Lemma \ref{lem:new_lemma_4.6}.

  After a change of variables $\kappa_x:(y,\overline{w})\mapsto
  v(x,y,\overline{w})=(v_1(x,y,\overline{w}),\overline{v_2}(x,y,\overline{w}))$ for which the phase is
  the holomorphic extension of the standard quadratic form $-|v|^2$,
  one has, by definition,
  \[
    \Delta_{\kappa_x}=\Delta_v=\sum_{i=1}^{d}\cfrac{\partial^2}{\partial v_{1,i}\partial\overline{v_{2,i}}}.
  \]
  Hence, denoting the inverse change of variables by $(x,v)\mapsto
  (x,y(x,v),\overline{w}(x,v))$, we obtain
  \begin{multline*}
   \left.\cfrac{\Delta^n_{\kappa_x}}{n!}\left(u_{0}(y)b_{k+1-n}(x,y,\overline{w})\right)\right|_{(y,\overline{w})=(x,\overline{y}_c(x))}\\
    =\sum_{|\mu|=n}\sum_{\nu\leq 2\mu}\cfrac{n!(2\mu)!}{\mu!\nu!(2\mu-\nu)!}\,\partial^{\nu}_{v}u_l(y(x,v))|_{v=0}\partial^{2\mu-\nu}_vb_{k+1-n-l}(x,y(x,v),\overline{w}(x,v))|_{v=0}.
  \end{multline*}
  Since at most $n$ derivatives on $u_l$ appear in \eqref{eq:rec_nice_form} by the
  first step, in the
  expression above, the differential operator
  \[\partial^{\nu}_vu_l(y(x,v,\overline{v}))_{v=0}\] can be replaced with its
  truncation into a differential operator of degree less or equal to
  $n$, which we denote by $(\nabla^{\nu}_{\kappa})^{[\leq n]}u_l(x)$ as
  in \cite{deleporte_toeplitz_2018}, Lemma 4.6.
  In particular, for every $\rho\in \N^d$,
  \begin{multline*}
    \nabla^{\rho}_x\Delta_v^n[u_l(y(v,\overline{v}))b_{k+1-n-l}(x,y(v,\overline{v}),\overline{w}(v,\overline{v}))]_{v=0}=\\
    \sum_{|\mu|=n}\sum_{\nu\leq 2\mu}\sum_{\rho_1\leq
      \rho}\cfrac{n!(2\mu)!\rho!}{\mu!\nu!(2\mu-\nu)!\rho_1!(\rho-\rho_1)!}\,\nabla^{\rho_1}_x(\nabla^{\nu}_{\kappa})^{[\leq n]}u_l(x)\nabla^{\rho-\rho_1}_x\nabla^{2\mu-\nu}_vb_{k+1-n-l}(x,y(x,v,\overline{v}),\overline{w}(x,v,\overline{v}))_{v=0}.
  \end{multline*}
  Moreover, if $|\mu|=n$ then
  \[\cfrac{n!}{\mu!}\leq (2d)^n,
  \]
  and if $\nu\leq 2\mu$ then, by Lemma 2.4 in \cite{deleporte_toeplitz_2018},
  \[
    \cfrac{(2\mu)!\rho!}{\nu!(2\mu-\nu)!\rho_1!(\rho-\rho_1)!}=\binom{2\mu}{\nu}\binom{\rho}{\rho_1}\leq \binom{2n}{|\nu|}\binom{|\rho|}{|\rho_1|}.\]
  Hence,
  \begin{multline*}
    \|\nabla^j_x\Delta_v^n[u_l(y(v,\overline{v}))b_{k+1-n-l}(x,y(v,\overline{v}),\overline{y}(v,\overline{v}))]_{v=x=0}\|_{\ell^1}\\
    \leq
    (2d)^n\sum_{i_1=0}^{2n}\sum_{j_1=0}^j\binom{2n}{i_1}\binom{j}{j_1}\|\nabla^{j_1}_x(\nabla^{i_1}_{\kappa})^{[\leq
      n]}u_l|_{x=0}\|_{\ell^1}\|b_{k+1-n-l}\circ \kappa^{-1}\|_{C^{j-j_1+2n-i_1}(W_1)}.
  \end{multline*}
  By the induction hypothesis, one has
  \[
    \|\nabla^ju_l(0)\|_{\ell^1}\leq
    C_u\cfrac{r^jR^lj!l!}{(j+l+1)^m},
    \]
then, by Lemma \ref{lem:new_lemma_4.6}, there exists a
fixed $C_{\kappa}>0$ such that
  \begin{multline}
    \label{eq:ctrl-der-u}
    \|\nabla^{j_1}_x(\nabla^{i_1}_{\kappa})^{[\leq n]}u_l(y(v,\overline{v}))_{x=v=0}\|_{\ell^1}\\
    \leq
    i_1^{d+1}j_1^{d+1}C_u\cfrac{r^{j_1+i_1}R^ll!}{(i_1+j_1+l+1)^m}(C_{\kappa})^{i_1}\times \begin{cases}\max((n+j_1)!(i_1-n)!\;,\;j_1!i_1!)&\text{if
      }i_1\geq n\\(i_1+j_1)!&\text{else}.
    \end{cases}
  \end{multline}
  
  If $j_1+\min(i_1,n)\geq l$, one has the more precise control
   \begin{multline}
    \label{eq:ctrl-der-u-prec}
    \|\nabla^{j_1}_x(\nabla^{i_1}_{\kappa})^{[\leq n]}u_l(y(v,\overline{v}))_{x=v=0}\|_{\ell^1}\\
    \leq
    i_1^{d+1}j_1^{d+1}C_u\cfrac{(r/4)^{j_1+i_1}R^l}{(i_1+j_1+l+1)^{m-\frac
      12}}(C_{\kappa})^{i_1}\times \begin{cases}\max((n+j_1+l)!(i_1-n)!\;,\;(j_1+l)!i_1!)&\text{if
      }i_1\geq n\\(i_1+j_1+l)!&\text{else}.
    \end{cases}
  \end{multline}
 In the case $l=0$, the constant $C_u$ can be replaced with the smaller constant
 $C_0$.

 Let us now control $\lambda_k$ using equation
    \eqref{eq:Induction-uk} at $x=0$:
    \[
      \lambda_k=\left.\sum_{n=2}^{k+1}\sum_{l=0}^{k+1-n}\cfrac{\Delta_{\kappa_0}^n}{n!}(u_l(y)b_{k+1-n-l}(0,y,\overline{w}))\right|_{(y,\overline{w})=(0,0)}.
    \]
    Then, by the induction hypothesis, \eqref{eq:ctrl-der-u}, and the
    fact that
    \[
      \|\nabla^{j_1}(b_l\circ \kappa^{-1})\|_{\ell^1}\leq
      C_b\frac{r^{j_1}R^lj_1!(l-1+\mathds{1}_{l=0})!}{(1+j_1+l)^m},
    \]
    we obtain
    \begin{multline*}
      |\lambda_k|\leq
      C_uC_b\sum_{n=2}^{k+1}\cfrac{R^kk!}{(k+2)^m}(2d)^nR\left(\cfrac{C_{\kappa}^2r^2}{R}\right)^n\times\\\sum_{i_1=0}^{2n}\frac{(2n)!A(i_1,0,n)}{i_1!n!k!}\left(\sum_{l=0}^{k-n}\cfrac{(k-n-l)!l!(k+2)^m}{(i_1+l+1)^m(k+2+n-l-i_1)^m}+\cfrac{(k+1-n)!(k+2)^m}{(i_1+k-n+2)^m(1+2n-i_1)^m}\right),
    \end{multline*}
    with
    \[
      A(i_1,j_1,n)=\begin{cases}\max((n+j_1)!(i_1-n)!\;,\;j_1!i_1!)&\text{if
        }i_1\geq n\\(i_1+j_1)!&\text{else};
      \end{cases}
    \]
    in the sum above, we separated the case $l=k+1-n$, corresponding
    to the specific control on $b_0$.
    
    For $l\leq k-n$, one has
    \[
      \cfrac{(2n)!l!A(i_1,0,n)(k-n-l)!}{i_1!n!k!}\leq
      4^n.\]
    Indeed, in this case where $j_1=0$, one has always
    $n!(i_1-n)!\leq i_1!$ if $i_1\geq n$, so that
    $A(i_1,0,n)=i_1!$ in all cases. We obtain
    \[
      \cfrac{(2n)!l!(k-n-l)!}{n!k!}\leq
      \binom{2n}{n}\frac{n!l!(k-n-l)!}{k!}\leq 4^n.
    \]

    In the specific case $l=k-n+1$, one has similarly
    \[
      \frac{(2n)!(k+1-n)!}{n!k!}\leq 4^n\frac{n!(k+1-n)!}{k!}
    \]
    and the right-hand side is a log-convex function of $n$. At $n=2$
    we obtain
    \[
      32\frac{(k-1)!}{k!}\leq 6^2,
    \]
    and at $n=k+1$,
    \[
      4^{k+1}(k+1)\leq 6^{k+1},
    \]
    so that one has always
    \[
      \frac{(2n)!(k+1-n)!}{n!k!}\leq 6^n.
    \]
    Getting back to the control on $\lambda_k$, we obtain
    \[
      |\lambda_k|\leq C_uC_b\sum_{n=2}^{k+1}\cfrac{R^kk!}{(k+2)^m}(2d)^nR\left(\cfrac{6C_{\kappa}^2r^2}{R}\right)^n\sum_{i_1=0}^{2n}\sum_{l=0}^{k+1-n}\cfrac{(k+2)^m}{(i_1+l+1)^m(k+2+n-l-i_1)^m},
    \]
    Since $(k+2)^m\leq (k+2+n)^m$, one has
    \[
      |\lambda_k|\leq C_uC_b\sum_{n=2}^{k+1}\cfrac{R^kk!}{(k+2)^m}(2d)^nR\left(\cfrac{6C_{\kappa}^2r^2}{R}\right)^n\sum_{i_1=0}^{2n}\sum_{l=0}^{k+1-n}\cfrac{(k+n+2)^m}{(i_1+l+1)^m(k+2+n-l-i_1)^m},
    \]

Then, by Lemma 2.13 in \cite{deleporte_toeplitz_2018}, there holds
\[
  |\lambda_k|\leq
  C_uC_b\cfrac{R^kk!}{(k+2)^m}R\sum_{n=2}^{k+1}\left(\cfrac{12dC_{\kappa}^2
      r^2}{R}\right)^n.
\]

If $R$ is large enough (once $r,m,C_u,C_{\lambda}$ are fixed), then
one can conclude:
\[
  |\lambda_k|\leq C_{\lambda}\frac{R^kk!}{(k+2)^m}.
  \]
  
We now pass to the control on $u_k$. We recall that $u_k$ solves an
equation of the form
\[
  X\cdot u_k=hu_k+g_k,
\]
with $X$ and $h$ independent on $k$ and
\[
g_k:x\mapsto-\sum_{l=1}^{k-1}\lambda_lu_{k-l}(x)-\lambda_ku_0(x)+
  \sum_{n=2}^{k+1}\sum_{l=0}^{k+1-n}\left.\frac{\Delta_{\kappa_x}}{n!}(u_l(y)b_{k+1-n-l}(x,y,\overline{w}))\right|_{(y,\overline{w})=(x,\overline{y}_c(x))}.
\]

We want to prove
\begin{equation}\label{eq:ctrl_g_4.2}
  \|\nabla^jg_k(0)\|_{\ell^1}\leq \epsilon C_u \frac{r^jR^k(j+1)!k!}{(j+k+2)^m}
\end{equation}
and, if $j\geq k$, the more precise control
\begin{equation}\label{eq:ctrl_g_sharp_4.2}
  \|\nabla^jg_k(0)\|_{\ell^1}\leq \epsilon
  C_u\frac{(r/4)^jR^k(j+k+1)!}{(j+k+2)^{m-\frac 12}},
\end{equation}
in order to apply
Proposition \ref{prop:transport}. Here $\epsilon>0$ must be smaller
than $\frac{1}{C(h,b_0,\varphi)}$ in Proposition \ref{prop:transport},
in order to conclude the induction and prove the claimed controls on $u_k$.

One has first
\[
  \|\lambda_k\nabla^ju_0(0)\|_{\ell^1}\leq
  C_{\lambda}C_0\cfrac{(r/4)^jR^kj!k!}{(j+1)^m(k+2)^m}
\]
Once $C_{\lambda}$ and $\epsilon$ are fixed, one has $C_{\lambda}C_0\leq \epsilon C_u$ for
$C_u$ large enough. In particular, one has, for all $j$ and $k$,
\[
  \|\lambda_k\nabla^ju_0(0)\|_{\ell^1}\leq
  \epsilon C_u\cfrac{(r/4)^jR^kj!k!}{(j+1)^{m-\frac 12}(k+2)^m}\leq \epsilon
  C_u\cfrac{r^jR^kj!k!}{(j+k+2)^m},
\]
and for $j\geq k$,
\[
  \|\lambda_k\nabla^ju_0(0)\|_{\ell^1}\leq
  \epsilon C_u\cfrac{(r/4)^jR^kj!k!}{(j+1)^{m-\frac 12}(k+2)^m}\leq \epsilon
  C_u\cfrac{(r/4)^jR^k(j+k)!}{(j+k+2)^{m-\frac 12}}.
\]
Moreover, for all $j$,
\[
  \left\|\sum_{l=1}^{k-1}\lambda_l\nabla^ju_{k-l}(0)\right\|_{\ell^1}\leq
  C_{\lambda}C_u\cfrac{r^jR^{k}j!k!}{(j+k+2)^m}\sum_{l=1}^{k-1}\underbrace{\cfrac{l!(k-l)!}{k!}}_{=\binom{k}{l}^{-1}\leq
    1}\cfrac{(k+j+2)^m}{(l+2)^m(k-l+j+1)^m}.
\]
Hence, by Lemma 2.13 in \cite{deleporte_toeplitz_2018},
\[
  \left\|\sum_{l=1}^{k-1}\lambda_l\nabla^ju_{k-l}(0)\right\|_{\ell^1}\leq CC_{\lambda}C_u\frac{3^m}{4^m}\frac{r^jR^{k}k!j!}{(j+k+2)^m}.
\]
Once $C_{\lambda}$ and $C_u$ are fixed, the constant
$CC_{\lambda}C_u\frac{3^m}{4^m}$ is smaller than $\epsilon C_u$ for
$m$ large enough (and $r,R$ large enough accordingly), and we obtain
\[
  \left\|\sum_{l=1}^{k-1}\lambda_l\nabla^ju_{k-l}(0)\right\|_{\ell^1}\leq
  \epsilon C_u\frac{r^jR^kj!k!}{(j+k+2)^m}.
  \]

If in addition $j\geq k$, then in particular $j\geq k-l$ for all $1\leq l \leq
k-1$, so that one has the more precise control
\[
  \left\|\sum_{l=1}^{k-1}\lambda_l\nabla^ju_{k-l}(0)\right\|_{\ell^1}\leq
  C_{\lambda}C_u\cfrac{(r/4)^jR^{k}(j+k)!}{(j+k+2)^{m-\frac 12}}\sum_{l=1}^{k-1}\underbrace{\cfrac{l!(k-l+j)!}{(j+k)!}}_{=\binom{k+j}{l}^{-1}\leq
    1}\cfrac{(k+j+2)^{m-\frac 12}}{(l+2)^{m-\frac 12}(k-l+j+1)^{m-\frac 12}}.
\]
Again, by Lemma 2.13 in \cite{deleporte_toeplitz_2018}, we obtain, for
$m$ large enough,
  \[
  \left\|\sum_{l=1}^{k-1}\lambda_l\nabla^ju_{k-l}(0)\right\|_{\ell^1}\leq
  CC_{\lambda}C_u\left(\frac 34\right)^{m-\frac 12}\frac{(r/4)^jR^{k}(k+j)!}{(j+k+2)^{m-\frac
    12}}\leq
  \epsilon C_u\frac{(r/4)^jR^{k}(k+j)!}{(j+k+2)^{m-\frac 12}}.
\]
It remains to estimate
\[
  \left\|\nabla^j\left[x\mapsto \sum_{n=2}^{k+1}\sum_{l=0}^{k+1-n}\left.\cfrac{\Delta^n_{\kappa_x}}{n!}(u_l(y)b_{k+1-n-l}(x,y,\overline{w}))\right|_{(y,\overline{w})=(x,\overline{y}_c(x))}\right]_{x=0}\right\|_{\ell^1}.
\]
Let us first suppose $j\leq k$.
By \eqref{eq:ctrl-der-u}, and since
\[
  \|\nabla^{j_1}(b_l\circ \kappa^{-1})\|_{\ell^1}\leq
  \frac{(r/2)^{j_1}R^lj_1!l!}{(j+l+1)^m},
\]
one has
\begin{multline*}
  \left\|\nabla^j\left[x\mapsto \sum_{n=2}^{k+1}\sum_{l=0}^{k+1-n}\left.\cfrac{\Delta^n_{\kappa_x}}{n!}(u_l(y)b_{k+1-n-l}(x,y,\overline{w}))\right|_{(y,\overline{w})=(x,\overline{y}_c(x))}\right]_{x=0}\right\|_{\ell^1}\\
  \leq
  C_uC_b\cfrac{r^jR^k(j+1)!k!}{(j+k+2)^m}\sum_{n=2}^{k+1}R\left(\cfrac{C_{\kappa}^2
      r^2}{R}\right)^n\sum_{l=0}^{k+1-n}\sum_{i_1=0}^{2n}\sum_{j_1=0}^j\\
  \cfrac{(2n)!j!l!A(i_1,j_1,n)(k-n-l+1)!(2n-i_1+j-j_1)!}{2^{2n-i_1+j-j_1}i_1!(2n-i_1)!j_1!(j-j_1)!n!(j+1)!k!}\\\times
  \cfrac{(k+j+2)^m}{(i_1+l+j_1+1)^m(k+2+n-l-i_1+j-j_1)^m}.
\end{multline*}
Let us prove, similarly to the control on $\lambda_k$, that
\[
   \cfrac{(2n)!l!A(i_1,j_1,n)(k-n-l+1)!(2n-i_1+j-j_1)!}{2^{2n-i_1+j-j_1}i_1!(2n-i_1)!j_1!(j-j_1)!n!k!(j+1)}\leq
   16^n.
 \]

 First of all,
 \[
   \frac{(2n-i_1+j-j_1)!}{(2n-i_1)!(j-j_1)!}\leq 2^{2n-i_1+j-j_1},
 \]
 so we are left with
 \[
   \cfrac{(2n)!l!A(i_1,j_1,n)(k-n-l+1)!}{i_1!j_1!n!k!(j+1)}.
 \]
 Suppose first $i_1\leq n$, so that $A(i_1,j_1,n)=(i_1+j_1)!$. We are
 left with trying to bound
 \[
   \cfrac{(2n)!l!(i_1+j_1)!(k-n-l+1)!}{i_1!j_1!n!k!(j+1)}.
 \]
 This is increasing with respect to $i_1$ and $j_1$, so that this is
 smaller than
 \[
   \cfrac{(2n)!l!(n+j)!(k-n-l+1)!}{n!n!j!k!(j+1)}\leq 4^n\frac{l!(n+j)!(k-n-l+1)!}{k!(j+1)!}.
 \]
 The right-hand side is log-convex with respect to $l$, and it is
 equal, at the boundaries $l=0$ and $l=k+1-n$, to
 \[
   4^n\frac{(k+1-n)!(n+j)!}{k!(j+1)!}.
 \]
 
  This is a log-convex function of $n$, which varies from $2$ to
  $k+1$. At $n=2$ we obtain $4^2\frac{j+2}{k}\leq 16^2$ (since $j\leq k$).
  At $n=k+1$, we obtain instead
  \[
    4^{k+1}\frac{(k+j+1)!}{(j+1)!k!}\leq 4^{k+1}2^{k+1+j}\leq 16^{k+1},
  \]
  since $j\leq k$. Hence, for all $n$ it is smaller than $16^n$.
 
 If now $i_1\geq n$, and if $A(i_1,j_1,n)=j_1!i_1!$, then we must
 simply bound
 \[
   \frac{(2n)!l!(k-n-l+1)!}{n!k!(j+1)}\leq 4^n\frac{l!n!(k-n-l+1)!}{k!(j+1)}.
 \]
 With respect to $l$, the right-hand side reaches a maximum at $l=0$
 and $l=k-n+1$, yielding
 \[
   4^n\frac{n!(k-n+1)!}{k!(j+1)}.
 \]
This log-convex function of $n$ is equal to $4^2\frac{2}{k(j+1)}\leq
16^2$ at $n=2$, and at $n=k+1$ we obtain
\[
  4^{k+1}\frac{k+1}{j+1}\leq 16^{k+1};
\]
thus, again, it is smaller than $16^n$ in all cases.

 To conclude, if $i_1\geq n$ and $A(i_1,j_1,n)=(j_1+n)!(i_1-n)!$, then it remains
 to bound
 \[
   \frac{(2n)!l!(i_1-n)!(j_1+n)!(k-n-l+1)!}{i_1!j_1!n!k!(j+1)}\leq \frac{(2n)!l!(i_1-n)!(j_1+n)!(k-n-l+1)!}{i_1!j_1!n!k!(j+1)}.
 \]
 This function is increasing with respect to $j_1$ and decreasing with
 respect to $i_1$, so that it is maximal at $i_1=n,j_1=j$, where we
 obtain
 \[
   \frac{(2n)!l!(j+n)!(k-n-l+1)!}{n!(j+1)!n!k!}\leq 4^n\frac{l!(j+n)!(k-n-l+1)!}{(j+1)!k!}
   \]
 which we bounded a few lines above.
In conclusion,
\begin{multline*}
  \left\|\nabla^j\left[x\mapsto \sum_{n=2}^{k+1}\sum_{l=0}^{k+1-n}\left.\cfrac{\Delta^n_{\kappa_x}}{n!}(u_l(y)b_{k+1-n-l}(x,y,\overline{w}))\right|_{(y,\overline{w})=(x,\overline{y}_c(x))}\right]_{x=0}\right\|_{\ell^1}\\
  \leq
  C_uC_b\frac{r^jR^k(j+1)!k!}{(j+k+2)^m}\sum_{n=2}^{k+1}R\left(\cfrac{16C_{\kappa}^2
      r^2}{R}\right)^n\sum_{l=0}^{k+1-n}\sum_{i_1=0}^{2n}\sum_{j_1=0}^j
  \cfrac{(k+j+2)^m}{(i_1+l+j_1+1)^m(k+2+n-l-i_1+j-j_1)^m}\\
  \leq
  C_uC_b\frac{r^jR^k(j+1)!k!}{(j+k+2)^m}\sum_{n=2}^{k+1}R\left(\cfrac{16C_{\kappa}^2
      r^2}{R}\right)^n\sum_{l=0}^{k+1-n}\sum_{i_1=0}^{2n}\sum_{j_1=0}^j
  \cfrac{(k+j+n+2)^m}{(i_1+l+j_1+1)^m(k+2+n-l-i_1+j-j_1)^m}.
\end{multline*}
By Lemma 2.13 in \cite{deleporte_toeplitz_2018}, there exists $C>0$
such that, for $m$ large enough, (and $r,R$ large enough accordingly) one has
\begin{multline*}
   \left\|\nabla^j\left[x\mapsto \sum_{n=2}^{k+1}\sum_{l=0}^{k+1-n}\left.\cfrac{\Delta^n_{\kappa_x}}{n!}(u_l(y)b_{k+1-n-l}(x,y,\overline{w}))\right|_{(y,\overline{w})=(x,\overline{y}_c(x))}\right]_{x=0}\right\|_{\ell^1}\\
  \\\leq
  CC_uC_b\frac{r^jR^k(j+1)!k!}{(j+k+2)^m}\sum_{n=2}^{k+1}R\left(\frac{16C_{\kappa}^2r^2}{R}\right)^n.
\end{multline*}
Thus, for $R$ large enough (once $C_u,C_{\lambda},m,r$ are fixed),
\[
  \left\|\nabla^j\left[x\mapsto
      \sum_{n=2}^{k+1}\sum_{l=0}^{k+1-n}\left.\cfrac{\Delta^n_{\kappa_x}}{n!}(u_l(y)b_{k+1-n-l}(x,y,\overline{w}))\right|_{(y,\overline{w})=(x,\overline{y}_c(x))}\right]_{x=0}\right\|_{\ell^1}\leq
  \epsilon C_u\frac{r^jR^k(j+1)!k!}{(j+k+2)^m}.
\]
This concludes the proof of the control \eqref{eq:ctrl_g_4.2}.

Suppose now that $j\geq k$. We start again from
\begin{multline*}
  \left\|\nabla^j\left[x\mapsto
      \sum_{n=2}^{k+1}\sum_{l=0}^{k+1-n}\left.\cfrac{\Delta^n_{\kappa_x}}{n!}(u_l(y)b_{k+1-n-l}(x,y,\overline{w}))\right|_{(y,\overline{w})=(x,\overline{y}_c(x))}\right]_{x=0}\right\|_{\ell^1}\\\leq\sum_{n=2}^{k+1}(2d)^n\sum_{l=0}^{k+1-n}\sum_{i_1=0}^{2n}\sum_{j_1=0}^j\frac{(2n)!j!}{i_1!(2n-i_1)!j_1!(j-j_1)!n!}\|\nabla_x^{j_1}(\nabla_{\kappa}^{i_1})^{[\leq
    n]}u_l|_{x=0}\|_{\ell^1}\|b_{k+1-n-l}\circ \kappa^{-1}\|_{C^{j-j_1+2n-i_1}(W_1)}.
\end{multline*}

We decompose the sum into two parts, corresponding to
$j_1+\min(i_1,n)<l$ and $j_1+\min(i_1,n)\geq l$.

In the first part,
the control on $u_l$ is the same as previously: one has
\[
  \|\nabla_x^{j_1}(\nabla_{\kappa}^{i_1})^{[\leq n]}u_l|_{x=0}\|\leq
  C_u(C_{\kappa})^{i_1}\frac{r^{j_1+i_1}R^ll!A(i_1,j_1,n)}{(i_1+j_1+l+1)^m},
\]
and
\[
  \|b_l\circ \kappa^{-1}\|_{C^{j-j_1+2n-i_1}(W_1)}\leq C_b\frac{(r/8)^{j_1}R^{l}j_1!l!}{(j_1+l+1)^m},
  \]
so that
\begin{multline*}
  \sum_{n=2}^{k+1}(2d)^n\sum_{l=0}^{k+1-n}\sum_{i_1=0}^{2n}\sum_{j_1=0}^{l-\min(i_1,n)}
 \frac{(2n)!j!}{i_1!(2n-i_1)!j_1!(j-j_1)!n!}\|\nabla_x^{j_1}(\nabla_{\kappa}^{i_1})^{[\leq
   n]}u_l|_{x=0}\|_{\ell^1}\|b_{k+1-n-l}\circ \kappa^{-1}\|_{C^{j-j_1+2n-i_1}(W_1)}\\
 \leq C_uC_b\frac{(r/4)^jR^k(j+k+1)!}{(j+k+2)^{m-\frac
     12}}\sum_{n=2}^{k+1}\left(\frac{2dC_{\kappa}^2r^2}{R}\right)^n\sum_{l=0}^{k+1-n}\sum_{i_1=0}^{2n}\sum_{j_1=0}^{l-\min(i_1,n)}\\\frac{4^{i_1+j_1}(2n)!j!l!A(i_1,j_1,n)!(j-j_1+2n-i_1)!(k+1-n-l)!}{2^{j-j_1+2n-i_1}n!i_1!(2n-i_1)!j_1!(j-j_1)!(j+k+1)!}\\
 \times \frac{(j+k+2)^{m-\frac
   12}}{(j+k+2-n-l-j_1-i_1)^m(i_1+j_1+l+1)^m}.
\end{multline*}
Let us now prove that
 \[
   4^{j_1+i_1}\frac{j!(2n)!l!A(i_1,j_1,n)(k-n-l+1)!(2n-i_1+j-j_1)!}{2^{2n-i_1+j-j_1}n!i_1!(2n-i_1)!j_1!(j-j_1)!(j+k+1)!}\leq
   256^n\sqrt{j}.
 \]
 First, as before
 \[
   \frac{(2n-i_1+j-j_1)!}{(2n-i_1)!(j-j_1)!}\leq
   2^{2n-i_1+j-j_1},
 \]
 and we obtain
 \[
   4^{j_1+i_1}\frac{j!(2n)!l!A(i_1,j_1,n)(k-n-l+1)!}{n!i_1!j_1!(j+k+1)!}.
 \]
 If $i_1\leq n$, then $A(i_1,j_1,n)=(i_1+j_1)!$, and we obtain
 \[
   4^{i_1+j_1}\frac{j!(2n)!l!(i_1+j_1)!(k-n-l+1)!}{i_1!j_1!n!(j+k+1)!}\]
This quantity is increasing with respect to $j_1$, so that it is
maximal at $j_1=l-i_1$, yielding
\[
  4^l\frac{j!(2n)!l!l!(k-n-l+1)!}{i_1!(l-i_1)!n!(j+k+1)!}
\]
Suppose first $n\geq \frac l2$. Then, with respect to $i_1$, this
quantity reaches a maximum at $i_1=\frac l2$, and we obtain
\[
  4^l\binom{l}{\frac{l}{2}}\binom{2n}{n}\frac{j!n!l!(k-n-l+1)!}{(j+k+1)!}\leq
  4^n8^l\frac{j!n!l!(k-n-l+1)!}{(j+k+1)!}\leq 256^n.
\]

Suppose next $n\leq \frac l2$. Then, with respect to $i_1$, the
maximum of
\[
  4^l\frac{j!(2n)!l!l!(k-n-l+1)!}{i_1!(l-i_1)!n!(j+k+1)!}
\]
is reached at $i_1=n$, yielding
\[
  4^{l}\binom{2n}{n}\frac{j!l!l!(k-n-l+1)!}{(l-n)!(j+k+1)!}\leq
  4^{l+n}\frac{j!l!l!(k-n-l+1)!}{(l-n)!(j+k+1)!}.
\]
This decreasing function of $k$
reaches its maximum at $k=n+l-1$ (the minimal value for $k$ for
$n,l,j$ fixed). We obtain
\[
  4^{n+l}\frac{j!l!l!}{(l-n)!(j+l+n)!}\leq 4^n\left(4^l\frac{j!l!l!}{(l-n)!(j+l+n)!}\right).
\]
To conclude, the quantity inside parentheses
is a decreasing function of $n$ ; at $n=0$, we obtain
\[
  4^l\frac{j!l!}{(j+l)!}\leq 2\sqrt{j},
\]
since $l\leq j$. Thus, we can bound the original quantity by
$4^{n+\frac{1}{2}}\sqrt{j}\leq 256^n\sqrt{j}$.

If $i_1\geq n$ and $A(i_1,j_1,n)=i_1!j_1!$, it remains to bound
\[
  4^{i_1+j_1}\frac{j!(2n)!l!(k-n-l+1)!}{n!(j+k+1)!}.
\]
Again, this decreasing function of $k$ is maximal at $k=l+n-1$,
yielding
\[
  4^{i_1+j_1}\frac{j!(2n)!l!}{n!(j+n+l)!}\leq
  4^{2n}4^{j_1}\frac{j!(2n)!l!}{n!(j+n+l)!}\leq
  4^{n+l}\frac{j!(2n)!l!}{n!(j+n+l)!}\leq
  16^n4^l\frac{j!n!l!}{(j+n+l)!}.
\]
Now
\[
  4^l\frac{j!n!l!}{(j+n+l)!}
\]
is a decreasing function of $n$, and at $n=0$ it is equal to
\[
  4^l\frac{j!l!}{(j+l)!}\leq \sqrt{j}.
\]
Hence, in this case the original quantity is bounded by $16^n\sqrt{j}$.

If $i_1\geq n$ and $A(i_1,j_1,n)=(i_1-n)!(j_1+n)!$, we have to bound
\[
4^{i_1+j_1}\frac{j!(2n)!l!(i_1-n)!(j_1+n)!(k-n-l+1)!}{i_1!j_1!n!(j+k+1)!}.
\]
This quantity is decreasing with respect to $k$, and at the minimal
value $k=l+n-1$ it is equal to 
\[
  4^{i_1+j_1}\frac{j!(2n)!l!(i_1-n)!(j_1+n)!}{i_1!j_1!n!(j+n+l)!}.
\]
This is now increasing with respect to $j_1$, and at the maximal value
$j_1=l-n$, it is equal to
\begin{align*}
  4^{i_1+l-n}\frac{j!(2n)!l!(i_1-n)!l!}{i_1!(l-n)!n!(j+n+l)!}&\leq
                                                                               4^{i_1-n}\frac{\binom{2n}{n}}{\binom{i_1}{n}}4^l\frac{j!l!l!}{(l-n)!(j+n+l)!}\\
  &\leq 16^n4^l\frac{j!l!l!}{(l-n)!(j+l+n)!}.
\end{align*}
We proved above that
\[
  4^l\frac{j!l!l!}{(l-n)!(j+l+n)!}\leq \sqrt{j},
\]
and we obtain that the original quantity is bounded by
$16^n\sqrt{j}$.

We thus obtain
\begin{multline*}
  \sum_{n=2}^{k+1}(2d)^n\sum_{l=0}^{k+1-n}\sum_{i_1=0}^{2n}\sum_{j_1=0}^{l-\min(i_1,n)}
 \frac{(2n)!j!}{i_1!(2n-i_1)!j_1!(j-j_1)!n!}\|\nabla_x^{j_1}(\nabla_{\kappa}^{i_1})^{[\leq
   n]}u_l|_{x=0}\|_{\ell^1}\|b_{k+1-n-l}\circ \kappa^{-1}\|_{C^{j-j_1+2n-i_1}(W_1)}\\
 \leq
 C_uC_b\frac{(r/4)^jR^k(j+k+1)!}{(j+k+2)^{m-\frac 12}}\sum_{n=2}^{k+1}\left(\frac{512dC_{\kappa}^2r^2}{R}\right)^n\sum_{l=0}^{k+1-n}\sum_{i_1=0}^{2n}\sum_{j_1=0}^{l-\min(i_1,n)}\frac{\sqrt{j}(j+k+2)^{m-\frac
   12}}{(j+k+2-n-l-j_1-i_1)^m(i_1+j_1+l+1)^m}.
\end{multline*}
Since $\sqrt{j}\leq\sqrt{j+k+2}$, one can apply Lemma 2.13 in
\cite{deleporte_toeplitz_2018} and obtain, for $R$ large enough,
\begin{multline*}
  \sum_{n=2}^{k+1}(2d)^n\sum_{l=0}^{k+1-n}\sum_{i_1=0}^{2n}\sum_{j_1=0}^{l-\min(i_1,n)}
 \frac{(2n)!j!}{i_1!(2n-i_1)!j_1!(j-j_1)!n!}\|\nabla_x^{j_1}(\nabla_{\kappa}^{i_1})^{[\leq
   n]}u_l|_{x=0}\|_{\ell^1}\|b_{k+1-n-l}\circ \kappa^{-1}\|_{C^{j-j_1+2n-i_1}(W_1)}\\ \leq
 \epsilon C_u\frac{(r/4)^jR^l(j+k+1)!}{(j+k+2)^{m-\frac 12}}.
\end{multline*}
If $j_1+\min(i_1,n)\geq l$, then
the control on $u_l$ takes the form
\[
  \|\nabla_x^{j_1}(\nabla_{\kappa}^{i_1})^{[\leq n]}u_l|_{x=0}\|_{\ell^1}\leq
  (C_{\kappa})^{i_1}C_u\frac{(r/4)^jR^lB(i_1,j_1,l,n)}{(1+j_1+i_1+l)^{m-\frac 12}}
\]
with
\[
  B(i_1,j_1,l,n)=
  \begin{cases}
    \max((n+j_1+l)!(i_1-n)!,\,(j_1+l)!i_1!)&\text{ if }i_1\geq n\\
    (i_1+j_1+l)! &\text{ else.}
  \end{cases}
\]
Together with
\[
  \|b_{l}\circ \kappa^{-1}\|_{C^{j_1}(W_1)}\leq C_b\frac{(r/8)^{j_1}R^lj_1!l!}{(j_1+l+1)^m},
\]
we obtain
\begin{multline*}
  \sum_{n=2}^{k+1}(2d)^n\sum_{l=0}^{k+1-n}\sum_{i_1=0}^{2n}\sum_{j_1=l-\min(i_1,n)}^{j}
 \frac{(2n)!j!}{i_1!(2n-i_1)!j_1!(j-j_1)!n!}\|\nabla_x^{j_1}(\nabla_{\kappa}^{i_1})^{[\leq
   n]}u_l|_{x=0}\|_{\ell^1}\|b_{k+1-n-l}\circ \kappa^{-1}\|_{C^{j-j_1+2n-i_1}(W_1)}\\ \leq
 C_uC_b\frac{(r/4)^jR^k(j+k+1)!}{(j+k+2)^{m-\frac
     12}}\sum_{n=2}^{k+1}\left(\frac{2dC_{\kappa}^2r^2}{R}\right)^n\sum_{l=0}^{k+1-n}\sum_{i_1=0}^{2n}\sum_{j_1=l-\min(i_1,n)}^{j}\\
 \frac{(2n)!j!B(i_1,j_1,l,n)(k+1-n-l)!(j-j_1+2n-i_1)!}{2^{2n-i_1+j-j_1}i_1!(2n-i_1)!j_1!(j-j_1)!n!(j+k+1)!}\\
   \times \frac{(j+k+2)^{m-\frac 12}}{(1+j_1+i_1+l)^{m-\frac 12}(2+k+n+j-l-j_1-i_1)^{m}}.
 \end{multline*}
 Let us prove that, in this sum, one has always
 \[
   \frac{(2n)!j!B(i_1,j_1,l,n)(k+1-n-l)!(j-j_1+2n-i_1)!}{2^{2n-i_1+j-j_1}i_1!(2n-i_1)!j_1!(j-j_1)!n!(j+k+1)!}\leq
   4^n.
 \]
 First,
 \[
   \frac{(j-j_1+2n-i_1)!}{(j-j_1)!(2n-i_1)!}\leq 2^{2n-i_1+j-j_1},
 \]
 and it remains to bound
 \[
   \frac{(2n)!j!B(i_1,j_1,l,n)(k+1-n-l)!}{i_1!j_1!n!(j+k+1)!}.
 \]
 If $i_1\leq n$, we obtain
 \[
   \frac{(2n)!j!(i_1+j_1+l)!(k+1-n-l)!}{i_1!j_1!n!(j+k+1)!}.
 \]
 This quantity is increasing with respect to $i_1$ and $j_1$, so that
 it is maximal at $i_1=n$ and $j_1=j$, where we obtain
 \[
   \frac{(2n)!(n+j+l)!(k+1-n-l)!}{n!n!(j+k+1)!}=\frac{\binom{2n}{n}}{\binom{j+k+1}{j+n+l}}\leq
   4^n.
 \]
 If $i_1\geq n$ and $B(i_1,j_1,l,n)=i_1!(j_1+l)!$, we obtain
 \[
   \frac{(2n)!j!(j_1+l)!(k+1-n-l)!}{j_1!n!(j+k+1)!}.
 \]
 This increasing function of $j_1$ reaches a maximum at $j_1=j$, where
 we obtain
 \[
   \frac{(2n)!(j+l)!(k+1-n-l)!}{n!(j+k+1)!}=\frac{\binom{2n}{n}}{\binom{j+k+1}{n,j+l}}\leq 4^n.
 \]
 If $i_1\geq n$ and $B(i_1,j_1,l,n)=(i_1-n)!(j_1+l+n)!$, then we
 obtain
 \[
   \frac{(2n)!j!(i_1-n)!(j_1+l+n)!(k+1-n-l)!}{j_1!n!i_1!(j+k+1)!}.
 \]
 This is an increasing function of $j_1$, as well as a decreasing
 function of $i_1$, so that it is maximal at $i_1=n,j_1=j$, where we
 obtain again
 \[
   \frac{(2n)!(j+l+n)!(k+1-n-l)!}{n!n!(j+k+1)!}=\frac{\binom{2n}{n}}{\binom{j+k+1}{j+l+n}}\leq
   4^n.
 \]
As before, we conclude using Lemma 2.13 in
\cite{deleporte_toeplitz_2018}; if $m,r,R$ are
large enough, then we obtain
\begin{multline*}
  \sum_{n=2}^{k+1}(2d)^n\sum_{l=0}^{k+1-n}\sum_{i_1=0}^{2n}\sum_{j_1=l-\min(i_1,n)}^{j}
 \frac{(2n)!j!}{i_1!(2n-i_1)!j_1!(j-j_1)!n!}\|\nabla_x^{j_1}(\nabla_{\kappa}^{i_1})^{[\leq
   n]}u_l|_{x=0}\|_{\ell^1}\|b_{k+1-n-l}\circ \kappa^{-1}\|_{C^{j-j_1+2n-i_1}(W_1)}\\ \leq
 \epsilon C_u\frac{(r/4)^jR^l(j+k+1)!}{(j+k+2)^{m-\frac 12}}.
\end{multline*}

This concludes the proof of \eqref{eq:ctrl_g_sharp_4.2}. Now, we can
apply Lemma \ref{prop:transport}:
there exists $C(b_0,\varphi)$ such that
\[
  \|\nabla^ju_k(0)\|_{\ell^1}\leq \epsilon
  C(b_0,\varphi)C_u\cfrac{r^jR^k(j+k)!}{(j+k+1)^m}.
\]
If $\epsilon$ is chosen such that $\epsilon<C(b_0,\varphi)^{-1}$, one
can conclude the induction.

{\bfseries Third step.}

We successfully constructed and controlled the sequences
$(\lambda_k)_{k\geq 0}$
and $(u_k)_{k\geq 0}$ that satisfy \eqref{eq:eigval} at every order. Let us now
prove that $u_k$ is controlled on a small neighbourhood of $0$.

In the second step, we controlled the functions $u_k$ as follows,
\emph{at zero}:
\[
  \|\nabla^ju_k(0)\|_{\ell^1}\leq
  C_u\cfrac{r^jR^kj!k!}{(j+k+1)^m}.
\]
Since $u_k$ is real-analytic, in a small neighbourhood of
zero, it is given by the power series
\[
  u_k(y)=\sum_{\nu}\cfrac{\nabla^{\nu}u_k(0)}{\nu!}\,y^{\nu}.
\]
Since
\[
  \cfrac{\nabla^{\nu}u_k(0)}{\nu!}\leq
  C_uR^kk!\cfrac{|\nu|!}{\nu!}r^{|\nu|}\leq
  C_uR^kk!(rd)^{|\nu|},
\]
the power series above converges for $y\in P(0,(rd)^{-1})$, the
polydisk centred at zero with radius $(rd)^{-1}$. Moreover, for every
$a<1$, there exists $C(a)$ such that
\[
  \sup_{P(0,a(rd)^{-1})}|u_k|\leq C(a)C_uR^kk!.
\]
In particular, by Proposition 2.14 in \cite{deleporte_toeplitz_2018}, for every $a<\frac 12$,
there exists $C(a)$ such that
\[
  \|u_k\|_{H\left(-d,\frac{d^2r}{a},P(0,\frac{a}{rd})\right)}\leq C(a)C_uR^kk!.
\]
In other terms, letting $V=P(0,a(2rd)^{-1})$, for every $j\geq 0$, one
has
\[
  \|u_k\|_{C^j(V)}\leq
  C(a)C_u\cfrac{R^k(\frac{d^2}{a}r)^jj!k!}{(j+1)^{-d}}.
\]
In particular, $u$ is an analytic symbol on $V$.
\end{proof}

We are now in position to perform an analytic summation.

\begin{lem}\label{prop:almost-hol}
  Let $f$, $V$, $\varphi$, $(u_k)_{k\geq 0}$, be as in
  Proposition \ref{prop:Wells-WKB}. There exists $c'>0$, $c_0>0$ and $C>0$
  such that, for all $0\leq c<c_0$, for all $N\in \N$, with
  \[
    u(N)=\mathds{1}_V\psi_0^Ne^{N\varphi}\sum_{k=0}^{cN}N^{-k}u_k,
  \]
  one has
  \[
    \|(1-S_N)u(N)\|_{L^2(M)}\leq Ce^{-c'N}.
    \]
  \end{lem}
  \begin{proof}
    Let $R>0$ be as in Proposition \ref{prop:Wells-WKB}. There exists
    $C_u>0$ such that, for all $k\in \N$,
    \[
      \sup_V\|u_k\|\leq C_uR^kk!.
    \]
    In particular, by Proposition \ref{prop:anal-symb}, for all
    $c\leq\frac{e}{3R}$, the sum
    \[
      \sum_{k=0}^{cN}N^{-k}u_k
    \]
    is bounded uniformly with respect to $N$.

    Let now $W\subset \subset V$ be such that
    $0\in W$, and let $\chi:M\to [0,1]$ be a smooth function such
    that $\mathds{1}_{W}\leq \chi \leq \mathds{1}_{V}$.

    Since there exists $\epsilon>0$ and $C>0$ such that, for all $x\in
    V$, for all $N\in \N$,
    \[
      |\psi_0^N(x)e^{N\varphi(x)}|\leq CN^{d}e^{-\epsilon
        \dist(x,0)^2N},
    \]
    and since $1-\chi$ is supported outside $W$, then there exists $C>0$ and $c'>0$ such that, for all $N\in \N$,
    \[
      \|(1-\chi)u(N)\|_{L^2(M)}\leq Ce^{-c'N}.
    \]
    In particular, since $S_N$ is an orthogonal projection,
    \[
      \|S_N[(1-\chi)u(N)]\|_{L^2(M)}\leq Ce^{-c'N}.
    \]
    Now
    \[
      \overline{\partial}(\chi u(N))=(\overline{\partial}\chi) u(N)
    \]
    satisfies
    \[
      \|(\overline{\partial}\chi) u(N)\|_{L^2(M)}\leq Ce^{-c'N}
    \]
    because $\overline{\partial}\chi$ is supported outside $W$ as
    well.

    We conclude using the H\"ormander $\overline{\partial}$
    inequality (see for instance \cite{tian_set_1990}, Proposition
    1.1, or \cite{deleporte_low-energy_2019}, Proposition 2.3.3)
    \[
      \|(1-S_N)v\|_{L^2(M)}\leq N^{-\frac 12}\|\overline{\partial}v\|_{L^2(M)}.
    \]
    Hence
    \[
      \|(1-S_N)\chi u(N)\|_{L^2(M)}\leq Ce^{-c'N},
    \]
    and we can conclude:
    \[
      \|(1-S_N)u(N)\|_{L^2(M)}\leq \|(1-\chi)u(N)\|_L^2(M) +
      \|(1-S_N)\chi u(N)\|_{L^2(M)} + \|S_N(1-\chi)u(N)\|_{L^2(M)}\leq
      3Ce^{-c'N}.
      \]
  \end{proof}

\begin{prop}\label{prop:sum-WKB}
  Let $f$, $V$, $(u_k)_{k\geq 0}$, $(\lambda_k)_{k\geq 0}$ be as
  above. There exists $c>0$, $c'>0$ and $C>0$ such that, for every $N\in \N$,
\[
  \left\|\left(T_N(f)-\sum_{j=0}^{cN}N^{-j-1}\lambda_j\right)\left(\mathds{1}_V\psi_{0}^Ne^{N\varphi}\sum_{k=0}^{cN}N^{-k}u_k\right)\right\|_{L^2(M)}\leq
  Ce^{-c'N}.
\]
\end{prop}
\begin{proof}  
Let $c>0$ be small enough, so that one can apply Proposition
\ref{prop:anal-symb}: $\sum_{k=0}^{cN}N^{-k}u_k$,
$\sum_{j=0}^{cN}N^{-j-1}\lambda_j$ are bounded independently on
$N$. Let
\begin{align*}
  u(N)&=\mathds{1}_V\psi_0^Ne^{N\varphi}\sum_{k=0}^{cN}N^{-k}u_k\\
  \lambda(N)&=\sum_{j=0}^{cN}N^{-j-1}\lambda_j.
\end{align*}

Outside of $V$, our presumed quasimode $u(N)$ is $0$, and one has also
\[
  \left\|T_N(f)u_N\right\|_{L^2(M\setminus
    V)}=O(e^{-c'N}):
\]
indeed, since $\varphi$ is admissible, outside any open set $W\subset \subset
V$ such that $0\in W$, one has $|\psi_0^Ne^{-N\varphi}|\leq Ce^{-c'N}$
for some $c'>0$, and the Szeg\H{o} projector $S_N$ decays away from
the diagonal, so that $\|\mathds{1}_{M\setminus
  V}S_N\mathds{1}_{W}\|\leq Ce^{-c'N}$ as well.

Since $S_Nu(N)=u(N)+O(e^{-c'N})$ by Lemma \ref{prop:almost-hol}, we
now replace $S_NfS_Nu(N)-\lambda(N)u(N)$ with
$S_Nfu(N)-\lambda(N)u(N)$, and estimate the $L^2$ norm of the latter
on $V$.
By construction, on $V$, there holds
\begin{multline*}
  [\left(S_Nf-\lambda(N)\right)u(N)](x)\\
  =-\sum_{j=0}^{cN}\sum_{k=cN-j}^{cN}N^{-1-j-k}\psi_{0}^N(x)e^{N\varphi(x)}\lambda_ju_k(x)+\sum_{j+k\leq cN}N^{-1-j-k}\psi_{0}^N(x)e^{N\varphi(x)}R(j,k,N)(x),
\end{multline*}
where $R(j,k,N)$ is the remainder at order $cN-k-j$ in the stationary
phase Lemma applied to
\[
  N^{2d}\lambda_je^{-N\varphi(x)}\int_{y\in
    M}e^{-N\Phi_1(x,y,\overline{y},0)+N\varphi(y)}(u* b)_k(x,y,\overline{y})\dd
  y.
\]

Since $\lambda*u$ is an analytic symbol by Proposition
\ref{prop:anal-symb}, we have, for $c>0$ and $c'>0$ small enough,
\[
    \left\|\sum_{j=0}^{cN}\sum_{k=cN-j}^{cN}N^{-1-j-k}\lambda_ju_k\right\|_{L^{\infty}(V)}\leq Ce^{-c'N},
  \]
  so that
  \[
  \left\|\left(\sum_{j=0}^{cN}\sum_{k=cN-j}^{cN}N^{-1-j-k}\lambda_ju_k(x)\right)\psi_{0}^N(x)e^{N\varphi(x)}\right\|_{L^{2}(V)}\leq
  Ce^{-c'N}.
\]
The remainder $R(j,k,N)$ can be estimated using Proposition 3.13 in
\cite{deleporte_toeplitz_2018}. Indeed, let $r>0$ and $R>0$ be such
that $u\in S^{r,R}_4(V)$ and $b\in S^{r,R}_4(V)$. By Proposition
\ref{prop:anal-symb}, $u*b$ is an analytic symbol of the same class,
so that
\[
  \|(u*b)_k\|_{C^j(V)}\leq CC_uC_bR^kr^j(j+k)!\leq
  (CC_uC_b(2R)^kk!)(2r)^jj!.
\]
In particular, $(u*b)_k$ admits a holomorphic extension to a
$k$-independent complex
neighbourhood $\widetilde{V}$ of $V$, with
\[
  \sup_{\widetilde{V}}|(u*b)_k|\leq CC_uC_b(2R)^kk!.
\]
In particular, by Proposition 3.13 in \cite{deleporte_toeplitz_2018},
one has, for some $c_1>0$, that the remainder at order $c_1N$ in the
stationary phase Lemma applied to
\[
  N^{2d}\lambda_je^{-N\varphi(x)}\int_{y\in
    M}e^{-N\Phi_1(x,y,\overline{y},0)+N\varphi(y)}(u* b)_k(x,y,\overline{y})\dd
  y
\]
is smaller than $CC_uC_b(2R)^k(2R)^jj!k!e^{-c'N}$.
In particular, 
\[
  \left(\frac{1}{n!}\Delta_{\kappa_x}^n((u*b)_kJ)(y_c)\right)_{n}
\]
is an analytic symbol in a fixed class, with norm smaller than $C(2R)^kk!$.

If $j+k<\frac 12 cN$, we will compare $R(j,k,N)$ to the remainder at
order $c_1N$. If $j+k\geq \frac 12 cN$, we will compare $R(j,k,N)$ to
the remainder at order $0$.

Without loss of generality, $c<c_1$. Then, for all $j,k$ such that
$j+k<\frac 12 cN$, since the expansion in the stationary phase
\[
  \sum_{n=cN-j-k}^{c_1N}(n!N^{d +
    n})^{-1}\Delta_{\kappa_x}^n((u*b)_kJ)(y_c)
\]
corresponds to an analytic symbol, then by Lemma \ref{prop:anal-symb}
this sum is $O(e^{-c'N})$; thus if $j+k<c/2$ one has
\[
  R(j,k,N)\leq Ce^{-c'N}.
\]
If $\frac 12 cN<j+k<cN$, then, on one hand
\[
   N^{-1-j-k}\left|N^{2d}\lambda_je^{-N\varphi(x)}\int_{y\in
    M}e^{-N\Phi_1(x,y,\overline{y},0)+N\varphi(y)}(u* b)_k(x,y,\overline{y})\dd
  y\right|\leq C\left(\frac{2R}{N}\right)^{j+k}(j+k)!
\]
is smaller than $Ce^{-c'N}$ if $c$ is small enough; on the other hand,
again
\[
  \left(\frac{1}{n!}\Delta_{\kappa_x}^n((u*b)_kJ)(y_c)\right)_{n}
\]
is an analytic symbol in a fixed class (with norm smaller than
$C(2R)^kk!$), so that, by Proposition \ref{prop:anal-symb}, if $c$ is
small enough,
\[
  N^{d-1-j-k}\lambda_j\sum_{n=0}^{cN-j-k}\frac{1}{n!N^n}\Delta_{\kappa_x}^n((u*b)_kJ)(y_c)<C\left(\frac{2R}{N}\right)^{j+k}(j+k)!\leq Ce^{-c'N}.
\]
This concludes the proof.
      \end{proof}

\section{Spectral estimates at the bottom of a well}
\label{sec:spectr-estim-at}

\subsection{End of the proof of Theorem \ref{thr:WKB-Bottom}}
\label{sec:end-proof-theorem}

We now prove part 2 of Theorem \ref{thr:WKB-Bottom}.
Suppose that $\min(f)=0$ and that the minimal set of $f$ consists in a
finite-number of non-degenerate minimal points $P_1,\ldots,P_j$. At
each of these points $P_i$ with $1\leq i\leq j$, one can construct (see Proposition \ref{prop:sum-WKB}) a sequence $v_i(N)$ of
$O(e^{-c'N})$-eigenfunctions of $T_N(f)$. From Proposition \ref{prop:WKB-formal}, if $\mu$
denotes the Melin value (see Section 3.3 of
\cite{deleporte_low-energy_2016}), then, for every $1\leq i\leq j$ one
has \[T_N(f)v_i(N)=N^{-1}\mu(P_i)v_i(N)+O(N^{-2}).\] Moreover, from
Theorem B in \cite{deleporte_low-energy_2016}, for $\epsilon>0$ small, the number of
eigenvalues of $T_N(f)$ in the interval
$[0,\stackrel[1\leq i\leq j]{}{\min}\mu(P_i)+N^{-1}\epsilon]$ is
exactly the number of $i$'s such that $P_i$ minimises $\mu$.

Hence, any normalised sequence of ground states of $T_N(f)$ is
$O(Ne^{-c'N})=O(e^{-(c'-\epsilon)N})$-close to a linear combination of
those $v_i(N)$ whose associated well $P_i$ minimises $\mu$ (as the
spectral gap is of order $N^{-1}$ and the the $v_i(N)$'s are $O(e^{-c'N})$-eigenvectors). This
concludes the proof.

\subsection{Tunnelling}
\label{sec:tunnelling}

The main physical application of Theorem \ref{thr:WKB-Bottom} is the
study of the spectral gap for Toeplitz operators that enjoy a local
symmetry. Let us formulate a simple version of this result.

\begin{prop}\label{prop:tunn-lower}
  Suppose that $\min(f)=0$ and
  that the minimal set of $f$ consists of two non-degenerate critical
  points $P_0$ and $P_1$.
  Suppose further that these wells are \emph{symmetrical}: there
  exist neighbourhoods $U_0$ of $P_0$ and $U_1$ of $P_1$, and a
  $\omega$-preserving biholomorphism $\sigma:U_0\mapsto U_1$, such
  that $\sigma\circ f=f$.

  Then there exists $c>0$ and $C>0$ such that, for every $N\geq 1$,
  the gap between the two first eigenvalues of $T_N(f)$ is smaller
  than $Ce^{-c'N}$.
\end{prop}
\begin{proof}
  Near $P_0$, one can build a sequence of $O(e^{-c'N})$-eigenvectors
  as in Proposition \ref{prop:sum-WKB}, with $c>0$; near
  $P_1$ one can build another sequence of $O(e^{-c'N})$-eigenvectors. Since $M$ and $f$ are equivalent near $P_0$
  and near $P_1$, the associated sequences of eigenvalues are
  identical up to $O(e^{-c'N})$, and the approximate eigenvectors are orthogonal
  with each other since they have disjoint support, so
  that there are at least two eigenvalues in an exponentially small
  window near the approximate eigenvalue. As above (see Theorem B in
  \cite{deleporte_low-energy_2016}), there are no more
  than two eigenvalues in the window $[\min Sp(T_N(f)),\min
    Sp(T_N(f))+\epsilon N^{-1}]$, for $\epsilon$ small; hence the claim.
  \end{proof}

  Unfortunately, the actual spectral gap between two symmetrical wells
  cannot be recovered from
  Proposition \ref{prop:Wells-WKB} or the solution $\varphi$ of the
  Hamilton-Jacobi equation, apart from the upper bound \eqref{eq:gap}.
  \begin{prop}
    \label{prop:tunn-upper}
    Suppose that $\min(f)=0$ and that the minimal set of $f$ consists
    of two symmetrical wells. Let $\lambda_0$ and $\lambda_1$ denote
    the two first eigenvalues of $T_N(f)$ (with multiplicity), and let
    \[
      \alpha=\liminf_{N\to
        +\infty}\left(-N^{-1}\log(\lambda_1-\lambda_0)\right).
      \]
      Then $\alpha$ cannot be bounded from above in terms of the best possible constant $c'$ in
      Proposition \ref{prop:sum-WKB}, and moreover $\alpha$ is unrelated to the solution
      $\varphi$ of the Hamilton-Jacobi equation.
  \end{prop}
  \begin{proof}
    We first let $\chi:[-1,1]\mapsto \R$ be an even smooth function; we
    suppose that $\chi$ reaches its minimum only at $-1$ and $1$, with
    $\chi(-1)=0$ and $\chi'(-1)>0$. We
    consider the function $f=\chi\circ z$ on $\S^2$, where $z:\S^2\to
    [-1,1]$ is the height function. Then $f$ is invariant
    under a rotation around the vertical axis, so that $T_N(f)$ is
    diagonal in the natural spin basis (which consists of the
    eigenfunctions for $T_N(z)$). Since $\chi'(-1)>0$, $f$ has two
    global minima (the North and South pole) and they are elliptic
    points. Among the spin basis, the states
    that minimise the energy are the coherent states at the North and
    South poles, respectively; they have the same energy since $f$ is
    invariant under the symmetry $z\to -z$. In this
    setting the first eigenvalue is degenerate, and shared between two
    states which localise at either of the two non-degenerate wells;
    one has $\alpha=+\infty$.

    Let us give a formal solution to the Hamilton-Jacobi equation. In
    stereographic coordinates near one of the poles, the symbol reads
    $g(|r|^2)=g(r\overline{r})$ for some $g\in
    \C^{\infty}(\R,\R)$. The expression $g(rs)$ does not make sense if
    $rs$ is not a real number, but taking $s=0$ yields $g(r\times
    0)=0$.
    A formal solution of $\widetilde{g}(x,\partial \varphi)=0$ is thus
    given by $\varphi=0$. This corresponds indeed to the exponential
    decay of the exact ground states: $\varphi=0$ means that the
    ground state decays as fast as the coherent state (they actually
    coincide).

    In the system above, the formal solution of the Hamilton-Jacobi
    equation yields the correct decay rate. However, from the point of
    view of Proposition \ref{prop:sum-WKB}, one has
    $c'=0$: if $\chi$ is not real-analytic near $1$ we cannot hope to perform an analytic summation for the
    sequence $\lambda_i$ as in Proposition \ref{prop:sum-WKB}. To be
    more precise, the ground state is
    \[
      \frac{N+1}{\pi}\int_{\S^2}\chi(z(x))\left(\frac{1-z(x)}{2}\right)^{2N}\dd\mathrm{Vol}(x),
    \]
    so that, if $\chi$ is not real-analytic near $-1$, one cannot
    approximate $\lambda_0$ by an analytic symbol up to $O(e^{-c'N})$
    for some $c'>0$.

    We consider now a smooth
    perturbation of the function $\chi$ above: let $\chi_1:\R\mapsto [0,1]$
    be a smooth, non-zero function supported on a compact subset of  $[0,1)$. If we
    replace $\chi$ with $\chi+\chi_1$ in the previous discussion, we still get
    a symbol invariant under vertical rotation, so that it is diagonal in
    the spin basis. Since $(\chi+\chi_1)\circ z=\chi\circ z$ where the
    latter is smallest (near the poles), the two candidates for the
    ground state are still the coherent states associated with the
    North and South pole, for $N$ large enough (all other states have
    an energy gap of order at least $N^{-1}$). The Hamilton-Jacobi equation has the same formal
    solution. However, the two candidates for the ground state now
    have different energies, with an exponentially small but non-zero
    gap, of order $e^{-\alpha N}$. In fact, from
    \[
      \lambda_1-\lambda_0=\frac{N+1}{\pi}\int_{\S^2}[\chi_1(z(x))-\chi_1(-z(x))]\left(\frac{1+z(x)}{2}\right)^{2N}\dd\mathrm{Vol}(x),
    \]
    one obtains
    \[\alpha=-2\log\left(\frac{1+\max(\mathrm{supp}\chi_1)}{2}\right).\] Here, $\alpha $ can be made arbitrarily small
    by choosing $\chi_1$ with support arbitrarily close to $1$. In
    this case, we identified a family of Toeplitz operators with
    symmetrical wells, with identical (formal) admissible
    solution of the Hamilton-Jacobi equation, and identically $c'=0$, but such that one has
    possibly $\alpha=+\infty$ (if $\chi_1=0$) or $\alpha$ arbitrarily
    small.
  \end{proof}
  The counterexample proposed in the proof is not entirely
  satisfactory, because it is not real-analytic on the whole
  manifold. In fact, in the situation of Proposition
  \ref{prop:tunn-lower}, if $f$ is real-analytic everywhere, then
  there is a global symmetry $\sigma:M\to M$, whose square is the
  identity, and such that $\sigma\circ f =f$. However, what
  Proposition \ref{prop:tunn-upper} illustrates is that even if the
  solution of the Hamilton-Jacobi equation can be globally defined (as
  a section), the fact that one can perform analytic extensions only
  in a fixed, not necessarily large neighbourhood of the real set
  means that large errors may occur.

  Another possible obstruction comes from the fact that, contrary to
  the case of two symmetric wells for Schrödinger operators
  \cite{helffer_puits_1985}, the symmetry $\sigma:M\to M$ may not be
  quantizable. For instance, on the unit torus $M=\C/\Z^2$, consider $f$
  invariant under horizontal translation by $\frac 12$, and having two
  non-degenerate wells. Then one cannot quantize $\sigma$ and decompose $H_0(M,L)$ into odd
  and even sections (for the action of $\sigma$) if $N$ is an odd
  integer. The tunnelling rate $\alpha$ may actually be different in
  the odd and even case.




\bibliographystyle{abbrv}
\bibliography{main}
\end{document}